\documentclass[11pt]{amsart}
\usepackage[margin=1.5in]{geometry}
\usepackage{
 amsmath, 
 amsxtra, 
 amsthm, 
 amssymb, 
 etex, 
 mathrsfs, 
 mathtools, 
 tikz-cd, 
 bbm,
 xr,
 comment,
 enumitem}
\usepackage[toc]{appendix} 
\usepackage[all]{xy}
\usepackage{hyperref}
\usepackage{url}
\usepackage{tipa}
\usepackage{dsfont}
\usepackage{ textcomp }
\usepackage{stmaryrd} 
\usepackage{amssymb}
\usepackage{mathabx} 
\usepackage{amsmath} 
\usepackage{amscd}
\usepackage{amsbsy}
\usepackage{comment, enumerate}
\usepackage{color}
\usepackage{mathtools,caption}
\usepackage{tikz-cd}
\usepackage{longtable}
\usepackage[utf8]{inputenc}
\usepackage[OT2,T1]{fontenc}
\usepackage{changepage}
\usepackage{commath,enumerate}
\usetikzlibrary{backgrounds}

\DeclareMathOperator{\Gal}{Gal}
\DeclareMathOperator{\rank}{rank}
\DeclareMathOperator{\ord}{ord}

\DeclareMathOperator{\coker}{coker}

\DeclareMathOperator{\GL}{GL}
\DeclareMathOperator{\SL}{SL}

\DeclareMathOperator{\ev}{ev}
\DeclareMathOperator{\inc}{inc}

\DeclareMathOperator{\Id}{Id}
\newtheorem{theorem}{Theorem}[section]
\newtheorem*{theorem*}{Theorem}
\newtheorem{lemma}[theorem]{Lemma}
\newtheorem{conjecture}[theorem]{Conjecture}
\newtheorem{proposition}[theorem]{Proposition}
\newtheorem{corollary}[theorem]{Corollary}

\newtheorem{defn}[theorem]{Definition}

\numberwithin{equation}{section}
\newtheorem{lthm}{Theorem} 

\theoremstyle{remark}
\newtheorem{remark}[theorem]{Remark}
\newtheorem{example}[theorem]{Example}

\newcommand{\Supp}{\mathrm{Supp}}

\newcommand{\EE}{\mathbb{E}}

\newcommand\EatDot[1]{}

\newcommand{\cL}{\mathcal{L}}
\newcommand{\cX}{\mathcal{X}}
\newcommand{\QQ}{\mathbb{Q}}
\newcommand{\ZZ}{\mathbb{Z}}
\newcommand{\FF}{\mathbb{F}}

\newcommand{\Zp}{\mathbb{Z}_p}

\definecolor{Green}{rgb}{0.0, 0.5, 0.0}

\newcommand{\cC}{\mathcal{C}}

\newcommand{\Q}{\mathbb{Q}}

\newcommand{\Z}{\mathbb{Z}}

\newcommand{\Char}{\mathrm{Char}}
\newcommand{\Tr}{\mathrm{Tr}}

\newcommand{\In}{\mathrm{In}}
\newcommand{\Out}{\mathrm{Out}}

\newcommand{\Fl}{\mathbb{F}_l}
\newcommand{\Deck}{\mathrm{Deck}}

\newcommand{\cY}{\mathcal{Y}}
\newcommand{\Div}{\textup{Div}}
\newcommand{\Pic}{\textup{Pic}}

\newcommand{\BF}{\mathcal{BF}}

\newcommand{\VV}{\mathbb{V}}
\newcommand{\sR}{\mathscr{R}}
\newcommand{\Zl}{\ZZ_\ell}
\newcommand{\Ql}{\QQ_\ell}
\newcommand{\pic}{\textup{Pic}}

  \DeclareFontFamily{U}{wncy}{}
  \DeclareFontShape{U}{wncy}{m}{n}{<->wncyr10}{}
  \DeclareSymbolFont{mcy}{U}{wncy}{m}{n}
  \DeclareMathSymbol{\sha}{\mathord}{mcy}{"58}
  \DeclareMathSymbol{\zhe}{\mathord}{mcy}{"11}

\mathcode`l="8000
\begingroup
\makeatletter
\lccode`\~=`\l
\DeclareMathSymbol{\lsb@l}{\mathalpha}{letters}{`l}
\lowercase{\gdef~{\ifnum\the\mathgroup=\m@ne \ell \else \lsb@l \fi}}%
\endgroup

\makeatletter
\newcommand{\mylabel}[2]{#2\def\@currentlabel{#2}\label{#1}}
\makeatother

\usepackage{tikz,xcolor,hyperref}

\title{Iwasawa theory for abelian towers of digraphs}

\author[A. Lei]{Antonio Lei}
\address[Lei]{Department of Mathematics and Statistics\\University of Ottawa\\
150 Louis-Pasteur Pvt\\
Ottawa, ON\\
Canada K1N 6N5}
\email{antonio.lei@uottawa.ca}

\author[K. Müller]{Katharina Müller}
\address[Müller]{Institut für Theoretische Informatik, Mathematik und Operations Research, Universität der Bundeswehr München, Werner-Heisenberg-Weg 39, 85577 Neubiberg, Germany}
\email{katharina.mueller@unibw.de}

\begin{document}

\begin{abstract}
Let $p$ and $\ell$ be prime numbers, and $d\ge1$ an integer. We formulate and prove Iwasawa main conjectures of the Picard groups and Bowen--Franks groups in $\mathbb{Z}_p^d$-towers of digraphs. In particular, we relate the $\ell$ parts of these groups to certain $p$-adic $L$-functions defined using a voltage assignment. In the case where $\ell$ is not equal to $p$, we make use of the recent work of Bandini--Longhi to define the appropriate characteristic ideals. We also prove the growth of the $\ell$-part of these groups, generalizing classical results of Sinnott and Washington on ideal class groups of number fields. Finally, we introduce the concept of defect, which compare certain algebraic and analytic ranks related to Bowen--Franks groups and study their asymptotic behaviour in a $\mathbb{Z}_p^d$-tower. 
\end{abstract}

\maketitle

\section{Introduction}
Let $p$ be a fixed prime number and $d\ge1$ an integer.  Let $X$ be a finite connected undirected graph, and consider a sequence of covers of finite connected graphs
$$X = X_{0} \leftarrow X_{1} \leftarrow X_{2} \leftarrow \ldots \leftarrow X_{n} \leftarrow \ldots $$
for which the composition $X_{n} \rightarrow \ldots \rightarrow X_{0} = X$ is a Galois cover whose Galois group of deck transformations is isomorphic to $(\mathbb{Z}/p^{n}\mathbb{Z})^d$ for all $n \ge 1$. For every finite connected undirected graph $X$, there is a finite abelian group ${\rm Pic}^{0}(X)$, whose cardinality, denoted by $\kappa(X)$, is the number of spanning trees of $X$.  By analogy to the classical Iwasawa theory for ideal class groups of number fields, one can study the variation of ${\rm ord}_{p}(\kappa(X_{n}))$ as $n$ increases. 

Suppose that $d=1$. In \cite{vallieres,vallieres2,vallieres3}, it was shown that there exist $\mu_{p},\lambda_{p} \in \mathbb{Z}_{\ge 0}$ and $\nu_{p} \in \mathbb{Z}$ such that
\begin{equation} \label{p=l}
{\rm ord}_{p}(\kappa(X_{n})) = \mu_{p} \cdot p^{n} + \lambda_{p} \cdot n + \nu_{p}, 
\end{equation}
provided $n$ is large enough.  Let $\ell$ be another prime number distinct from $p$.  Analogously to the results of Sinnott and Washington on class groups, in \cite{leivallieres} it was shown that there exist $\mu_{\ell} \in \mathbb{Z}_{\ge 0}$ and $\nu_{\ell} \in \mathbb{Z}$ such that
\begin{equation} \label{pn=l}
{\rm ord}_{\ell}(\kappa(X_{n})) = \mu_{\ell} \cdot p^{n} + \nu_{\ell},
\end{equation}
provided $n$ is large enough.
For general $d\ge1$, \cite{DV} extended \eqref{p=l} proving that $\ord_p(\kappa(X_{n}))$ can be expressed as a polynomial in $p^n$ and $n$.

When $p=\ell$, the formula (\ref{p=l}) can be obtained either using the theory of Ihara zeta and $L$-functions or the structure theorem of finitely generated modules over the Iwasawa algebra $\mathbb{Z}_{p}\llbracket \mathbb{Z}_{p}\rrbracket$, and there is a main conjecture connecting both approaches (see \cite[Theorem 6.1]{kataoka} and \cite[Theorem 5.2 and Remark 5.3]{KM1}). One goal of this work is to study an Iwasawa main conjecture that connects the theory of zeta functions to the $l$-primary part of Picard groups, in the setting of $\Zp^d$-towers of directed graphs, which are sometimes referred as digraphs in the literature. In the case where $p\ne l$, we will make use of the recent work of Bandini--Longhi \cite{BandiniLonghi}, where a novel conceptual framework to study modules of $\mathbb{Z}_{\ell}\llbracket \mathbb{Z}_{p}^d \rrbracket$ is developed. Furthermore, we prove a Sinnott--Washington type theorem, generalizing \eqref{pn=l} to digraphs, allowing $d$ to be any integers $\ge1$.

In addition, we prove analogous results for the Bowen--Franks groups attached to a $\Zp^d$-tower of digraphs, first introduced in \cite{BF} (see Definition~\ref{def:BF}). These groups play an important role in symbolic dynamics, where digraphs are related to shift spaces, and Bowen--Franks groups are important invariants that can be used to detect equivalence and conjugacy; see \cite{franks} and \cite{lindmarcus} for more details. 

\subsection{Main results}
In order to explain the main results of this article, we introduce some notation. Let $p$ be a fixed prime number and $d\ge1$ an integer. Given a prime number $l$, which may or may not be equal to $p$, we define the Iwasawa algebra $\Lambda_l(\Gamma)=\displaystyle\varprojlim_n\Zl[\Gamma/\Gamma^{p^n}]$, where $\Gamma=\Zp^d$. Let $\Gamma^\vee$ denote the set of finite-order $\overline{\Q}$-valued character of $\Gamma$. Given an element $F\in\Lambda_l(\Gamma)$ and $\omega\in\Gamma$, we can evaluate $F$ at $\omega$, denoted by $\omega(F)\in\overline{\Ql} $ (after fixing an embedding $\overline{\QQ}\hookrightarrow\overline{\Ql}$). We can also define the characteristic ideal of a finitely generated module $M$ over $\Lambda_l(\Gamma)$, denoted by $\Char_{\Lambda_l(\Gamma)}M$ (see \S\ref{sec:IwAlg} where this is reviewed). 

Throughout, we fix a $\Zp^d$-tower of strongly connected digraphs $(X_n)_{n\ge0}$, in the sense of Definition~\ref{def:tower}. Such a tower can be represented by a $\Zp^d$-valued voltage assignment, which is a function $\alpha$ from the set of directed edges of $X:=X_0$ to $\Zp^d$ (see Lemma~\ref{lem:voltage}). We define two $p$-adic functions associated with this tower, namely, the $p$-adic zeta function $L_p(X,\alpha)$ and the $p$-adic Bowen--Franks function $\cL_p(X,\alpha)$ in Definition~\ref{def:MXalpha}. These functions are closely related to the archimedean zeta functions associated with the digraphs $X_n$, as demonstrated in Proposition~\ref{interpolation}. An interesting feature of these functions is that they can be regarded as elements of $\Lambda_l(\Gamma)$ for all primes $l$.

On the algebraic side, we define 
$\Pic_l(X_\infty)$ (resp. $\BF_l(X_\infty)$) as the inverse limits of the $l$-primary parts of the Picard groups (resp. Bowen--Franks groups), which are finitely generated $\Lambda_l(\Gamma)$-modules.
The first main theorem of this article proves the validity of the Iwasawa main conjectures linking these algebraic objects to their analytic counterparts.
\begin{lthm}[Theorem~\ref{thm:IMC}]\label{thmA}
    Let $(X_n)_{n\ge0}$ be a $\Zp^d$-tower of strongly connected digraphs. For all primes $\ell$, we have
   \[\Char_{\Lambda_l(\Gamma)}(\pic_l(X_\infty))=(L_p(X,\alpha)),\]
   and
   \[\Char_{\Lambda_l(\Gamma)}(\BF_l(X_\infty))=(\cL_p(X,\alpha)).\]
   Furthermore, ${\rm Pic}_{l}(X_{\infty})$ is a torsion $\Lambda_l(\Gamma)$-module. 
\end{lthm}
In the special case where $l=p$ and the graphs $X_n$ are undirected, the statement on Picard groups given by Theorem~\ref{thmA} has been proved in \cite[Theorem 6.1]{kataoka} and \cite[Theorem 5.2 and Remark 5.3]{KM1}.

We remark that $\BF_l(X_\infty)$ is not always $\Lambda_l(\Gamma)$-torsion; see Example~\ref{eg:unbounded BF}. When it is not torsion, the assertion on Bowen--Franks groups in Theorem~\ref{thmA} simply says $(0)=(0)$.

The second main result of the article is the following generalization of formula \eqref{pn=l} to $\Zp^d$-towers of digraphs.

\begin{lthm}[Theorem~\ref{thm:generalizedSinnott}]\label{thmB}
       Let $(X_n)_{n\ge0}$ be a $\Zp^d$-tower of strongly connected digraphs.  Let $l\ne p$ be a prime number. Then 
    \[
    \ord_l(|\Pic_l(X_n)_\mathrm{tor}|)=p^{nd}\mu_l(L_p(X,\alpha))+O(p^{(d-1)n}).
    \] If $\mathcal{L}_p(X,\alpha)(\omega)\neq 0$ for all but finitely many $\omega\in\Gamma^\vee$, then
    \[\ord_l(\vert \BF_l(X_n)_{\textup{tor}}\vert)=p^{nd}\mu_l(\mathcal{L}_p(X,\alpha))+O(p^{(d-1)n}).\]
\end{lthm}
It is possible to utilize the general theory of Sinnott modules developed in \cite{BandiniLonghi} to estimate the growth of $\ord_l(|\Pic_l(X_n)_\mathrm{tor}|)$ and $\ord_l(\vert \BF_l(X_n)_{\textup{tor}}\vert)$, in particular Theorem~4.7 of \textit{op. cit.} However, Theorem~\ref{thmB} gives a more precise estimate (see Remarks~\ref{rk:sinnott} and \ref{rk:comparesinnott}). Our proof builds on a multi-variable generalization of Sinnott’s method \cite{sinnott}, which analyzes the $l$-adic valuation of one-variable polynomials evaluated at $p$-power roots of unity; see Corollary~\ref{cor:formula}. This general result may be of independent interest and should have further applications beyond the present setting.

When $p=l$, the method of \cite{cuoco-monsky} can be readily generalized to study the variation of $\ord_p(|\Pic_l(X_n)_\mathrm{tor}|)$ and $\ord_p(|\Pic_l(X_n)_\mathrm{tor}|)$, in a manner analogous to the results of \cite{DV} and \cite{KM1}. Since this does not introduce substantially new ideas, and in order to keep the paper to a reasonable length, we do not pursue this direction here.

Given a digraph $X$, we define the algebraic Bowen--Franks rank $b(X)$ of $X$ as the rank of $\BF(X)$ as a $\ZZ$-module and its analytic Bowen--Franks rank $a(X)$ as the order of the pole of the zeta function at $1$. In light of the main conjecture given by Theorem~\ref{thmA} and the relation between $\cL_p(X,\alpha)$ and the zeta function, we expect that these two quantities should be intimately linked. We define the defect of $X$ as $\delta(X)=a(X)-b(X)$ and investigate the variation of $\delta(X_n)$ as $X_n$ varies inside a $\Zp^d$-tower. In particular, we prove:

\begin{lthm}[Theorem~\ref{thm.growth-delta}]\label{thmC}
    Let $(X_n)_{n\ge 0}$ be a $\Z_p^d$-tower of strongly connected digraphs. For all $n\ge1$,
    \[\delta(X_n)\ge \delta(X_{n-1}).\]    
\end{lthm}

In the special case where $d=1$, we prove that $\delta(X_n)$ is bounded independently of $n$ provided that $\cL_p(X,\alpha)\ne0$  (see Lemma~\ref{lem:bounded-defect}). In particular, combined with Theorem~\ref{thmC}, we deduce that $\delta(X_n)$ becomes a constant for $n$ sufficiently large. Furthermore, when $(X_n)_{n\ge0}$ arises from a constant voltage assignment $\alpha$, that is, the image of $\alpha$ consists of a single element, similar to the towers studied in \cite{constant-voltage}, we can give an explicit $n$ for which $\delta(X_n)$ becomes a constant (see Lemma~\ref{defec-constant-voltage}). In the case of isogeny graphs defined using supersingular elliptic curves, we can even show that $\delta(X_n)=\delta(X_0)$ for all $n\ge0$; see Theorem~\ref{cor.delta-isogeny}. Our proof relies on recent results of Codogni--Lido \cite{codogni-lido} which give an explicit relation between isogeny graphs and modular forms. 

The behaviour of the defect $\delta(X_n)$ resembles the Leopoldt defect for number fields: Along $\Z_p$-towers of number fields the Leopoldt defect is increasing and uniformly bounded. The Leopoldt conjecture predicts that it is always zero. In the setting of $\Z_p$-towers of digraphs, our calculations have led us to predict that the defect should be constant throughout the tower. Whether there is a relation between these two defects on a conceptual level is unknown to the authors at the time of writing and will be the subject of forthcoming research.  

\subsection*{Acknowledgement}
The authors thank Giulio Codogni, Guido Lido, Mike Newman and Daniel Vallières for very helpful discussions during the preparation of this article. AL's research is supported by the NSERC Discovery Grants Program RGPIN-2020-04259.

\section{Preliminaries on directed graphs}

\begin{defn}
A \textbf{directed graph} (or a \textbf{digraph}) $X$ consists of a set of vertices $V(X)$ and a set of (directed or oriented) edges $\EE(X)$, together with an incidence map 
$${\rm inc}:\EE(X) \rightarrow V(X) \times V(X)$$ 
denoted by $e \mapsto {\rm inc}(e) = (o(e),t(e))$, where $o(e)$ is called the origin (or the source) and $t(e)$ the terminus (or target) of the directed edge $e$, respectively.  

Given $v\in V(X)$, we define
\begin{align*}
    \In(v)&=\{e\in\EE(X):t(e)=v\},\\
    \Out(v)&=\{e\in\EE(X):o(e)=v\}.
\end{align*}
The \textbf{in-degree} and \textbf{out-degree} of $v$ are defined as the cardinality of $\In(v)$ and $\Out(v)$, respectively.
\end{defn}

Throughout this article, we assume that $\EE(X)$ and $V(X)$ are finite sets. The cardinality of $V(X)$ is called the \textbf{\textit{order}} of $X$.

\begin{defn}
An \textbf{undirected graph} $X$ consists of a digraph $X = (V(X),\EE(X))$ together with an inversion map
$${\rm inv}:\EE(X) \rightarrow \EE(X) ,$$
denoted by ${\rm inv}(e) = \bar{e}$, which satisfies
$$\bar{e} \neq e,\quad \bar{\bar{e}} = e, \quad\text{and}\quad o(\bar{e}) = t(e) $$
for all $e \in \EE(X)$.  The set of undirected edges of $X$ is given by $E(X) = \EE(X)/{\rm inv}$. 
\end{defn}

Informally, an edge of an undirected graph can be viewed as a pair of directed edges $\{e, \bar e\}$ in a digraph, oriented in opposite directions. In particular, an undirected graph can be considered as a special case of digraphs. Although we primarily study digraphs in this article, all of our results apply equally to undirected graphs through this correspondence.

\begin{defn}
  Let $X$ be a digraph. The \textbf{forgetful functor} sends $X$ to an undirected graph $X'$, with $V(X)=V(X')$ and  $\EE(X')=\EE(X)\bigsqcup\EE(X)'$, where the inversion map induces bijections $\EE(X)\to\EE(X)'$ and $\EE(X)'\to\EE(X)$.
\end{defn}

More concretely, for each directed edge $e$ in $X$, we add an extra edge $\bar e$ in $X'$ with opposite orientation as $e$. The set of undirected edges in $X'$ consist of $\{e,\bar e\}$ as $e$ runs through $\EE(X)$. In particular, each directed edge in $X$ gives rise to exactly one undirected edge in $X'$.

\begin{defn}
    Let $X$ be a digraph. A \textbf{path} of length $k$ in $X$ is a sequence of edges $e_1,\dots, e_k\in\EE(X)$  such that $t(e_i)=o(e_{i+1})$ for $i=1,\dots k-1$. We say that it is a path from $o(e_1)$ to $t(e_k)$. It is said to be a \textbf{directed cycle} if $o(e_1)=t(e_k)$.

    We say that $X$ is \textbf{strongly connected} if there is a path from $v$ to $v'$ in $X$ for all  $(v,v')\in V(X)\times V(X)$.  Let $X'$ be the undirected graph obtained from $X$ after applying the forgetful functor.  We say that the graph $X$ \textbf{weakly connected} if $X'$ is strongly connected.
\end{defn}

\begin{defn}\label{def:div}
    Let $X$ be a digraph. The divisor group $\Div(X)$ is the group of divisors of $V(X)$. For $v\in V(X)$, we write $d_v$ for the out-degree of $v$ and define
    \[P_v=d_v(v)v-\sum_{e\in \Out(v)}t(e)\in \Div(X).\]
    
    The \textbf{group of principal divisors} of $X$, denoted by $\Pr(X)$, is the subgroup of $\Div(X)$ generated by $\{P_v\mid v\in V(X)\}$. We define the \textbf{Picard group} of $X$ as $\pic(X)=\Div(X)/\Pr(X)$.\footnote{This is sometimes called the critical group.} 

    Let $\{v_1,\dots, v_r\}$ be an enumeration of $V(X)$. Let $D_X$ denote the \textbf{degree matrix} of $X$, which is defined as the $r\times r$ diagonal matrix with diagonal entries given by $d_{v_1},\dots, d_{v_r}$. Let $A_X$ denote the \textbf{adjacency matrix} of $X$, which is the $r\times r$ matrix whose $(i,j)$-entry is equal to the number of edges $e$ such that $\inc(e)=(v_j,v_i)$.

For a prime number $l$, we write $\pic_l(X)$ for $\pic(X)\otimes\Zl$.
\end{defn}

We have $\pic(X)\cong \ZZ^r/(D_X-A_X)\ZZ^r$ and $\pic_l(X)\cong \Zl^r/(D_X-A_X)\Zl^r$ as groups. 

\begin{remark}\label{rank-picard} 
Given a digraph $X$ and $v\in V(X)$, the \textit{reachable set} $R(v)$ is defined as $$\{u\in V(X):\exists \text{ a path from } v \text{ to } u\}.$$
A \textit{reach} of $X$ is a maximal reachable set. The algebraic and geometric multiplicities of the eigenvalue $0$ of $D_X-A_X$ are equal to the number of reaches in $X$; see \cite[Theorem~IV.6]{VL}. In particular, if $X$ is strongly connected, the rank of the matrix $D_X-A_X$ is equal to $r-1$, where $r$ is the order of $X$.
\end{remark}
\begin{defn}\label{def:BF}
    Let $X$ be a digraph of order $r$. We define the \textbf{Bowen--Franks operator} on $\Div(X)\cong \ZZ^r$ by the matrix $\mathrm{BF}_X=\textup{Id}-A_X$. The \textbf{Bowen-Franks group} of $X$, denoted by $\BF(X)$, is defined as the cokernel of $\mathrm{BF}_X$, i.e., $$\BF(X)=\Div(X)/\mathrm{BF}_X\cdot\Div(X)\cong\ZZ^r/(\textup{Id}-A_X)\ZZ^r.$$ We define the \textbf{algebraic Bowen--Franks rank} of $X$, denoted by $b(X)$, as the rank of $\BF(X)$ as a $\ZZ$-module. 
\end{defn}

\begin{defn}
Let $X$ and $Y$ be digraphs. 
A \textbf{morphism of digraphs} $
f:Y \to X$
is a map that induces two maps $ V(Y) \to V(X)$ and $\EE(Y)\to \EE(X)$ that commute with the incidence map. It is said to be a locally an isomorphism around $v\in V(Y)$ if it induces bijections from $\In(v)$ and $\Out(v)$ to $\In(f(v))$ and $\Out(f(v))$, respectively.

If there is a bijective morphism between two digraphs $X$ and $Y$, we say that they are \textbf{isomorphic}, and write $X\cong Y$.

We say that $Y/X$ is a covering if there is a surjective digraph morphism from $Y$ to $X$ that is locally an isomorphism around every $v\in V(Y)$. 

    If $Y/X$ is a covering of directed graphs with the projection map $\pi:Y\to X$, the group of \textbf{deck transformations} of $Y/X$, denoted by $\Deck(Y/X)$, is the group of digraph automorphisms $\sigma:Y\to Y$ such that $\pi\circ\sigma=\pi$. %
    
    A covering $Y/X$ is said to be \textbf{$d$-sheeted} if $d$ is a positive integer such that each element of $V(X)$ has $d$ pre-images in $V(Y)$.
    
Finally, We say that the covering $Y/X$ is \textbf{Galois}  if there exists an integer $d$ such that $Y/X$ is a $d$-sheeted covering and $\vert \textup{Deck}(Y/X)\vert =d$.
\end{defn}

\begin{remark}\label{rk:connectedness}
Let $Y/X$ be a $d$-sheeted covering, where $X$ is weakly connected. Suppose that there are $r$ weakly connected components in $Y$ and let $Y_0$ be one of these components. Then $Y_0/X$ is a $d/r$-sheeted Galois covering and the group $\Deck(Y/X)$ induces isomorphisms across all weakly connected components of $Y$ (see \cite[Proposition~2.4]{LM2}). In particular, $|\Deck(Y_0/x)|=d/r$ and there are $r!$ elements in $\Deck(Y/X)$ given by permutations of these components.  Thus, 
\[\vert \textup{Deck}(Y/X)\vert \ge r!|\Deck(Y_0/x)|^r>d\]
unless $d=r=2$. In other words, If $Y$ is not weakly connected, then $Y/X$ is not Galois unless $d=r=2$.

If $d=r=2$, it is possible that $Y/X$ is a Galois covering even if $Y$ is not weakly connected. Let $X$ be an arbitrary weakly connected digraph, and let $Y$ be the digraph consisting of the disjoint union of two copies of $X$. Then there is a natural covering $Y\to X$, whose group of deck transformations is the permutation of the two components. As the covering is $2$-sheeted, it is indeed Galois. 
\end{remark}

\begin{defn}
    Let $X$ be a digraph and $(G,\cdot)$ a group. A $G$-valued \textbf{voltage assignment} on $X$ is a function $\alpha:\EE(X)\rightarrow G$. In the case of undirected graph, we require that $\alpha(\bar e)=\alpha(e)^{-1}$ for all $e\in\EE(X)$.

    Given a $G$-valued voltage assignment on $X$ where $G$ is finite, we define the derived digraph $X(G,\alpha)$ whose vertices and directed edges are given by $\VV(X)\times G$ and $\EE(X)\times G$, respectively. {In addition, for all $\sigma\in G$ and $e\in\EE(X)$ with $\inc(e)=(s,t)$, we have $\inc(e,\sigma)=((s,\sigma),(t,\sigma\cdot\alpha(e)))$.}

\end{defn}

In the case where $X$ is an undirected graph, the condition $\alpha(\bar e)=\alpha(e)^{-1}$ ensures that the inversion map on $\EE(X)$ extends to an inverse map on $\EE(X)\times G$. In particular, $X(G,\alpha)$ can be regarded as an undirected graph.

\begin{defn}\label{def:tower}
    Let $X$ be a digraph and $d\ge1$ an integer. A \textbf{$\Zp^d$-tower} over $X$, is a sequence of digraph coverings
\[
X=X_0\leftarrow X_1\leftarrow X_2\leftarrow \cdots \leftarrow X_n\leftarrow\cdots
\]
such that for each $n\ge0$, the cover $X_n/X$ is Galois with Galois group isomorphic to $(\ZZ/p^n\ZZ)^d$.
\end{defn}

It follows from Remark~\ref{rk:connectedness} that if $X$ is weakly connected, then the digraphs $X_n$ in a $\Zp^d$-tower over $X$ are necessarily weakly connected. For our purposes, we will assume that all $X_n$ are strongly connected.

\begin{lemma}\label{lem:voltage}
Let $(X_n)$ be a $\Zp^d$-tower of strongly connected digraphs. There exists a $\Zp^d$-valued voltage assignment $\alpha$ on $X$ such that $X_n\cong X((\ZZ/p^n\ZZ)^d,\alpha_n)$, where $\alpha_n$ is the composition of $\alpha$ with the natural projection map $\Zp^d\to (\ZZ/p^n\ZZ)^d$.
\end{lemma}
\begin{proof}
    We follow the discussions \cite[\S2.4]{gonet-thesis} and \cite[\S13-14]{terras} to realize $X_n$ as a derived digraph. Fix a spanning tree $T$ of $X$ rooted at a chosen vertex $v_0$ (i.e., a subgraph that includes all vertices of $X$, containing no directed cycles and $v_0$ is of in-degree $0$ and every other vertex has an in-degree $1$; such a spanning tree exists as $X$ is strongly connected). Let $G_n=\Gal(X_n/X)$. The digraph $X_n$ contains $p^{nd}$ copies of $T$, indexed by the elements of $G_n$. Each vertex of $X_n$ can be labelled as $(v,g)$, where $v\in V(X)$ and $g\in G_n$. If $e\in\EE(X)$ goes from $u$ to $v$, we define $\alpha_n(e)$ to be $g_2-g_1$ if there is an edge in $X_n$ going from $(u,g_1)$ to $(v,g_2)$. One can check that this is well-defined and $X_n$ is isomorphic to the derived digraph $X(\alpha_n)$. Further, since we have the coverings $X_{n+1}\to X_n\to X$, we can check that $\alpha_{n+1}\equiv \alpha_n\mod p^n$. Thus, taking inverse limit gives rise to a $\Zp^d$-valued voltage assignment, as desired.
\end{proof}

From now on, we fix an integer $d\ge1$ and write $\Gamma$ for the topological group $\Zp^d$. The group operation on $\Gamma$ will be written multiplicatively, rather than additively. Further, we write $\Gamma_n:=\Gamma/\Gamma^{p^n}\cong (\ZZ/p^n\ZZ)^d$.

\begin{defn}
Let $X$ be a digraph. If $v\in V(X)$, we define $\cC(v)$ as the set of all directed cycles in $X$ that pass through $v$. If $\alpha$ is a $\Gamma$-valued voltage assignment on $X$ and $C$ is a directed cycle in $X$ that consists of the edges $\{e_1,\ldots,e_m\}$, we define $\alpha(C)$ to be the element $\alpha(e_1)\cdots\alpha(e_m)\in\Gamma$.    
\end{defn}

\begin{defn}
Let $Y/X$ be a covering of digraphs, equipped with the morphism $\pi:Y\to X$. If $P=(e_1,\dots,e_k)$ is a path in $X$, we say that a path $P'=(e_1',\dots, e_k')$ in $Y$ is a lift of $P$ if $\pi(e_i')=e_i$ for $i=1,\dots,k$.
\end{defn}

The following lemma gives a criterion on the strong connectivity of derived digraphs.
\begin{lemma}\label{lem:connected}
Let $X$ be a strongly connected digraph and $\alpha$ a $\Gamma$-valued voltage assignment on $X$. For an integer $n\ge1$, write $\alpha_n$ for the composition of $\alpha$ with the projection map $\Gamma\to\Gamma_n$. Let $X_n$ denote the derived digraph $X(\Gamma_n,\alpha_n)$. The digraph $X_n$ is strongly connected  for all $n$ if and only if $\{\alpha(C)\mid C\in \cC(v)\}$ topologically generates $\Gamma$ for all $v\in V(X)$.
\end{lemma}
\begin{proof}
Suppose that $\{\alpha(C)\mid C\in \cC(v)\}$ generates $\Gamma$ topologically for all $v\in V(X)$.
    Let $(v,\tau)$ and $(w,\tau')$ be two vertices in $X_n$. Let $P$ be  a path from $v$ to $w$ in $X$ (this $P$ exists as $X$ is assumed to be strongly connected). Further, $P$ induces a path in $X_n$ from $(v,\tau)$ to $(w,\tau \alpha_n(P))$. By the assumption on $\alpha$, we can find a directed cycle $C\in \cC(w)$ such that $\alpha_n(C)=\tau'(\tau\alpha_n(P))^{-1}$. This lifts to a path from $(w,\tau\alpha_n(P))$ to $(w,\tau')$. Thus, there is a path from $(v,\tau)$ to $(w,\tau')$. Hence, $X_n$ is strongly connected.

    Conversely, assume that $X_n$ is strongly connected. Let $\tau\in \Gamma_n$ and $v\in V(X)$. As $X_n$ is strongly connected, there exists a path from $(v,1)$ to $(v,\tau)$. The image of this path in $X$ is a directed cycle $C$ that passes $v$, with $\alpha_n(C)=\tau$. As this holds for an arbitrary choice of $n$ and $\tau$, it follows that the set $\{\alpha(C)\mid C\in \cC(v)\}$ topologically generates $\Gamma$.
\end{proof}

\section{Review on Iwasawa algebras}\label{sec:IwAlg}
In what follows, we write $\Phi_{p,d}=\ZZ[T_1,\ldots,T_d;\Zp]$ for the ring of polynomials in $T_1,\ldots,T_d$ whose coefficients are in $\Z$ and the exponents are in $\Zp$. In other words, the elements of $\Phi_{p,d}$ are $\ZZ$-linear combinations of monomials of the form
$T_1^{a_1}\cdots T_d^{a_d}$, where $a_{1},\ldots,a_{d} \in \mathbb{Z}_{p}$. Note that we can regard $\Phi_{p,d}$ as the ring $\mathbb{Z}[\mathbb{Z}_{p}^{d}]$.

Let $l$ be a prime number (which may or may not be equal to $p$), we fix an embedding $\iota_l:\overline\QQ\hookrightarrow\overline{\Ql}$. Let $\mu_{p^{\infty}}$ denote the set of $p$-power roots of unity in $\overline{\QQ}$. We denote by its image in $\overline{\Ql}$ under $\iota_l$ by the same image.

Recall that we write $\Gamma=\Zp^d$ and the group operation on $\Gamma$ is written multiplicatively. 
Let $\pi_i:\Gamma\to\Zp$ denote the projection map from $\Zp^d$ to the $i$-th component. Let $\Gamma^\vee$ denote the set of $\overline{\Q}_{l}^\times$-valued finite-order characters of $\Gamma$ (we suppress the dependency of $l$ in the notation for simplicity). We define $\Gamma_n^\vee$ similarly for all $n\ge1$. Note that $\Gamma^\vee=\bigcup_{n\ge1} \Gamma_n^\vee$. Let $\sR^{(l)}$ denote the orbits of $\Gamma^\vee$ under the action of the absolute Galois group $G_{\Ql}$.

Given any $\vec\zeta=(\zeta_1,\ldots,\zeta_d)\in \mu_{p^\infty}^d$, let
\[
\ev_{\vec\zeta}:\Phi_{p,d}\to \ZZ[\mu_{p^\infty}]
\]
be the evaluation map given $T_1^{a_1}\cdots T_d^{a_d}\mapsto \zeta_1^{a_1}\cdots \zeta_d^{a_d}$. 
We fix a topological generator $\gamma$ of $\Zp$. For $1\le i\le d$,  write $\gamma_i=(1,\dots, 1,\gamma,1,\dots,1)\in\Gamma$, where $\gamma$ appears in the $i$-th component.
If $\omega\in \Gamma^\vee$, we have the evaluation map $\ev_\omega=\ev_{(\omega(\gamma_1),\dots,\omega(\gamma_d))}$ on $\Phi_{p,d}$. If $F\in \Phi_{p,d}$, we shall write $\omega(F)$ for $\ev_{\omega}(F)$.

We define $\Lambda_{l}(\Gamma)$ as the Iwasawa algebra 
\[
\Lambda_{l}(\Gamma)=\varprojlim_n\Zl[\Gamma_n],
\]
where $\Gamma_n=\Gamma/\Gamma^{p^n}$ and the connecting maps are natural projections.

\begin{lemma}\label{lem:injectiverings}
    There is an injective homomorphism of rings
    \[
    \Phi_{p,d}\hookrightarrow\Lambda_{l}(\Gamma).
    \]
\end{lemma}
\begin{proof}
    Given any integer $n\ge1$, there is a natural injection
    \[
    \Phi_{p,d}/\left(T_1^{p^n}-1,\dots T_d^{p^n}-1\right)\cong\Z[\Gamma_n]\hookrightarrow \Zl[\Gamma_n],
    \]
    after identifying $T_i$ with a topological generator of $\pi_i(\Gamma)$. This gives an injection, 
    \[
    \varprojlim_n \Phi_{p,d}/\left(T_1^{p^n}-1,\dots,T_d^{p^n}-1\right)\hookrightarrow \varprojlim_n\Zl[\Gamma_n]=\Lambda_{l}(\Gamma).
    \]
    Since
    \[
    \bigcap\left(T_1^{p^n}-1,\dots,T_d^{p^n}-1\right)=0,
    \]
    there is a natural injection
    \[
    \Phi_{p,d}\hookrightarrow  \varprojlim_n \Phi_{p,d}/\left(T_1^{p^n}-1,\dots,T_d^{p^n}-1\right),
    \]
    from which the lemma follows.
\end{proof}

We recall that if $p=l$ and $M$ is a finitely generated torsion $\Lambda_p(\Gamma)$-module, there is a pseudo-isomorphism between $M$ and a direct sum of cyclic $\Lambda_p(\Gamma)$-modules $\bigoplus_i \Lambda_p(\Gamma)/(F_i)$ for some $F_i\in\Lambda_p(\Gamma)$ (see \cite[Chapitre VII, no. 4, Théorème 5]{Bourbaki_commutative_algebra}). The characteristic ideal of $M$ is defined as
\[
\Char_{\Lambda_p(\Gamma)}M=\left(\prod_{i}F_i\right)\subseteq\Lambda_p(\Gamma)
\]
When $M$ is not torsion over $\Lambda_p(\Gamma)$, we set $$\Char_{\Lambda_p(\Gamma)}M=0.$$
Finally, we recall that after sending $\gamma_i$ to $T_i$, the Iwasawa algebra $\Lambda_p(\Gamma)$ can be identified with the ring of power series $\Zp[[T_1-1,\dots,T_d-1]]$.
\subsection{Characteristic ideals when $l\ne p$}
Throughout this subsection, we assume that $p\ne l$ are distinct prime numbers. 
 We have 
$$\Lambda_l(\Gamma)=\prod_{[\omega]\in\sR^{(l)}}\Lambda_l(\Gamma)_{[\omega]},$$
where $\Lambda_l(\Gamma)_{[\omega]}\subseteq\Lambda_l(\Gamma)$ is  a ring isomorphic to $\Zl[\omega(\Gamma)]$ (see \cite[Theorem 2.1]{BandiniLonghi}).

Let $\omega\in\Gamma^\vee$.  We define $e_{\omega,n}$ as the idempotent of $\omega$ if $\omega\in\Gamma_n^\vee$, and $0$ otherwise. Set $e_{[\omega],n}=\sum_{\omega\in[\omega]}e_{\omega,n}$. This defines an element $e_{[\omega]}=(e_{[\omega],n})\in \varprojlim \Zl[\Gamma_n]=\Lambda_l(\Gamma)$.
Given a $\Lambda_l(\Gamma)$-module $M$ and $[\omega]\in\sR^{(l)}$, we write $M_{[\omega]}=e_{[\omega]}M$.

If $M=\varprojlim_n M\otimes\Zl[\Gamma_n]$, we have the decomposition
\begin{equation}
\label{eq:decomp-mod}    
M\cong \prod_{[\omega]\in\sR^{(l)}}M_{[\omega]};
\end{equation}
see \cite[(13)]{BandiniLonghi}.

We say that a $\Lambda_l(\Gamma)$-module is locally finitely generated if each $M_{[\omega]}$ is a finitely generated over $\Lambda_l(\Gamma)_{[\omega]}$.
For such $M$, we write
\[
M_{[\omega]}=\left(\Lambda_l(\Gamma)_{[\omega]}\right)^{\oplus r_{[\omega]}}\oplus\bigoplus_{j=1}^{s_{[\omega]}}\Lambda_l(\Gamma)_{[\omega]}/I_{[\omega],j},
\]
where $r_{[\omega]}$ and $s_{[\omega]}$ are nonnegative integers and $I_{[\omega],j}$ are nonzero ideals that are uniquely determined by $M$ and $\omega$.
The $[\omega]$-characteristic ideal of $M$ is defined as
\[
\Char_{\Lambda_l(\Gamma)_{[\omega]}}\left(M_{[\omega]}\right)=\begin{cases}
    \prod_{j=1}^{s_{[\omega]}}I_{[\omega],j}& \text{if }r_{[\omega]}=0,\\
    (0_{[\omega]})&\text{otherwise.}
\end{cases}
\]
The characteristic ideal of $M$ is defined as
\[
\Char_{\Lambda_l(\Gamma)}(M)=\left(\Char_{\Lambda_l(\Gamma)_{[\omega]}}\left(M_{[\omega]}\right)\right)_{[\omega]\in\sR^{(l)}}\subseteq\Lambda_l(\Gamma).
\]
Note that if $\omega$ is of order $p^n$, then $\Lambda_l(\Gamma)_{[\omega]}\cong\Zl[\zeta_{p^n}]$. As $\Ql(\zeta_{p^n})/\Ql$ is unramified, $I_{[\omega],j}$ is generated by a power of $l$. In other words, $\Char_{\Lambda_l(\Gamma)_{[\omega]}}(M_{[\omega]})$ is of the form $l_{[\omega]}^{t_{[\omega]}}$ when $r_{[\omega]}=0$ (see \cite[Remark~4.8]{BandiniLonghi}).

We conclude this section with the following lemma, which will be used later in the article.

\begin{lemma}
\label{isomorphism-group rings}
Let  $\sR_n^{(l)}$ denote the set of classes $[\omega]\in\sR^{(l)}$ such that $\omega\in\Gamma_n^\vee$.     We have the following ring isomorphisms
    \[\Lambda_l(\Gamma)/(T_1^{p^n}-1,T_2^{p^n}-1,\dots , T_d^{p^n}-1)\cong\Zl[\Gamma_n]\cong \bigoplus_{[\omega]\in\sR_n^{(l)}}\Zl[\omega(\Gamma_n)]  \]
   \end{lemma}
\begin{proof}
    See \cite[Remark 2.2 (4)]{BandiniLonghi}.
\end{proof}

\section{$p$-adic functions attached to a $\Zp^d$-tower}
The main goal of this section is to define the $p$-adic zeta function and $p$-adic Bowen--Franks function of a $\Zp^d$-tower and study their interpolation properties. We begin with a review on the archimedean counterpart of these functions.

\begin{defn}
    Let $X$ be a digraph. A cycle $C$ in $X$ is said to be \textbf{prime} if it is not the $m$-multiple of another cycle for some integer $m\ge2$. Two cycles are said to be \textbf{equivalent} if one is a cyclic permutation of the other. An equivalence class of prime cycles under this equivalence relation is called a \textbf{prime} of $X$. If $[C]$ is a prime, where $C$ is a representative of the equivalence class, the length of $[C]$, denoted by $|[C]|$, is defined as the number of edges in $C$.

    We define the \textbf{zeta function} of $X$ as 
    \[
    Z_X(u)=\prod_{[C]}(1-u^{\vert [C]\vert})^{-1}\in\ZZ[[u]],
    \] where the product runs over all primes of $X$. 
\end{defn}
\begin{remark}
    Our definition of the zeta function follows the one given in \cite{ks} (denoted by $Z^\circ(u)$ therein). It follows from Lemma 2.2 in \textit{op. cit.} that it is equal to $\det(1-u A_X)^{-1}$. When $X$ is an undirected graph, this is different from the classical Ihara zeta function, since each undirected edge is represented by two directed edges under our convention.
\end{remark}

\begin{defn}
    Let $Y/X$ be a Galois covering. Let $P$ be a path of length $r$ in $X$ from $u$ to $v$. Let $P'$ be the unique path of length $r$  in $Y$ from $(u,1)$ to $(v,g)$, where $g\in \Gal(Y/X)$. The element $g$ is called the \textbf{Frobenius}  automorphism of $P$.
\end{defn}

For the rest of this section, we fix a $\Zp^d$-tower $(X_n)_{n\ge0}$ of strongly connected digraphs over a digraph $X$. The voltage assignment given by Lemma~\ref{lem:voltage} will be denoted by $\alpha$.

\begin{defn}
    Let $\omega\in\Gamma_n^\vee$. We define the \textbf{$L$-function} for $X_n/X$ and $\omega$ as
    \[L(X_n/X,\omega,u)=\prod_{[C]}(1-\omega([C])u^{\vert [C]\vert})^{-1}\in\ZZ[\omega(\Gamma_n)][[u]],\]
    where the product runs over all primes of $X$ and $\omega([C])$ denotes the image of the Frobenius of $[C]$ in $\Gamma_n=\Gal(X_n/X)$ under $\omega$.
\end{defn}

The following theorem can be regarded as the \textit{Artin formalism} for zeta functions of digraphs.
\begin{theorem}
\label{Artin-formalism}
   The following equality holds:
    \[Z_{X_n}(u)=\prod_{\omega\in\Gamma_n^\vee}L(X_n/X,\omega,u).\]
\end{theorem}
\begin{proof}
Let $G=\Gal(X_n/X)$. Let $[C]$ be a prime of $X$. We denote by $D_{[C]}\subset G$ the subgroup generated by the Frobenius of $[C]$. Note that one requires $|D_{[C]}|$ distinct lifts of $[C]$ in order to obtain a cycle in $X_n$ consisting solely of lifts of $[C]$.

We have
    \begin{align*}
        \prod_{\omega\in{\Gamma_n^\vee}}L(X_n/X,\omega,u)&=\prod_\omega\prod_{[C]}(1-\omega([C])u^{\vert [C]\vert})^{-1}\\
        &=\prod_{[C]}\prod_{\omega\in (D_{[C]})^\vee}(1-\omega([C])u^{\vert[C]\vert})^{-\vert G/D_{[C]} \vert}\\
       &=\prod_{[C]}(1-u^{\vert [C]\vert \vert D_{[C]}\vert})^{-\vert G/D_{[C]} \vert}\\
       &=Z_{X_n}(u).
    \end{align*}
    The last equality follows from the following fact: Let $[C']$ be a prime of $X_n$. Then its image in $X$ is the $f$-th power of a prime $[C]$ of $X$ for some integer $f$. Note that $\vert [C']\vert=\vert [C]\vert f$ and $f=\vert D_{[C]}\vert$. Furthermore, there are exactly $\vert G/D_{[C]}\vert $ primes of $X_n$ that are mapped to $[C]^f$ in $X$. 
\end{proof}
The following discussion is inspired by \cite[Section 2]{ks}.
We fix $\omega\in\Gamma_n^\vee$. For any positive integers $r$ and $m$, we define
\[M_r^m(\omega)=\sum_{[C]} \omega([C])^m,\]
where the sum runs over all primes of $X$ of length $r$. We further define
\[N_r(\omega)=\sum_D\omega([D]),\]
where the sum runs over all cycles in $X$ (up to equivalence) of length $r$.

Note that 
\begin{equation}
    \label{relation-enumeration}
    N_v(\omega)=\sum_{mr=v}lM^m_r(\omega). 
\end{equation}
\begin{lemma}\label{lem:L-exp}
We have the following equality of power series in $u$:
    \[L(X_n/X,\omega,u)=\exp\left(\sum_{v\ge 1}\frac{1}{v}N_v(\omega)u^v\right).\]
\end{lemma}
\begin{proof}
    We have
    \begin{align*}
        \log(L(X/Y,\omega,u))&=-\sum_{[C]}\log(1-\omega([C])u^{\vert [C]\vert})\\
        &=\sum_{[C]}\sum_{m\ge 1} \frac{1}{m}\omega([C])^mu^{m\vert [C]\vert}\\
        &=\sum_{m\ge 1}\sum_{r\ge 1}\frac{1}{m}u^{rm}\sum_{\vert[C]\vert=r}(\omega[C])^m\\
        &=\sum_{m\ge 1}\sum_{r\ge 1}\frac{1}{m}M_r^m(\omega)u^{mr}\\
        &=\sum_{v\ge 1}\frac{1}{v}u^v\sum_{rm=v}rM_r^m(\omega)\\
        &=\sum_{v\ge 1}\frac{1}{v}N_v(\omega)u^v,
    \end{align*}
    where the last equation follows from \eqref{relation-enumeration}.
\end{proof}

\begin{lemma}
\label{description-asmatrix}Let $A_w=\omega(A_\alpha)\in \textup{Mat}_{r\times r}(\Z[\omega(\Gamma)])$. Then 
    \[L(X_n/X,\omega,u)=\det(\mathrm{Id}-A_\omega u)^{-1}.\]
\end{lemma}
\begin{proof}
The matrix-entry $(A_\omega^v)_{i,j}$ is equal to $\sum_{[P]}\omega([P])$, where the sum runs over equivalence classes of all paths of length $v$ from $v_j$ to $v_i$. In particular, $\textup{Tr}(A_\omega^v)=N_v(\omega)$. 
    If $\lambda_1,\dots ,\lambda_r$ are the eigenvalues of $A_\omega$ counted with multiplicity, then  $$ N_v(\omega)=\sum_{i=1}^r\lambda_i^v.$$ Combined with Lemma~\ref{lem:L-exp}, we deduce
    \begin{align*}
        L(X/Y,\omega,u)&=\exp\left( \sum_{v\ge 1}\frac{1}{v}N_v(\omega)u^v\right) \\
        &=\exp\left( \sum_{v\ge 1}\frac{1}{v}\sum_{i=1}^N\lambda_i^vu^v\right)\\
        &=\prod_{i=1}^r\exp(-\log(1-\lambda_i u))=\prod_{i=1}^r\frac{1}{1-\lambda_i u}\\
        &=\det(1-A_\omega u)^{-1},
    \end{align*}
    as desired.
\end{proof}

\begin{defn}\label{def:MXalpha}
 We define the \textbf{$p$-adic zeta function} of a $\Zp^d$-tower of strongly connected digraphs over $X$, with voltage assignment $\alpha$, by
    \[
    L_p(X,\alpha):=\det(D_X-A_\alpha)\in\Phi_{p,d},
    \]
    where $D_X$ is the diagonal degree matrix given as in Definition~\ref{def:div} and $A_\alpha$ is the $r\times r$ matrix whose $(i,j)$-entry is given by $\displaystyle\sum_e\prod_{i=1}^dT_i^{\pi_i\circ\alpha(e)}$, where the sum runs over all directed edges from $v_j$ to $v_i$.

Similarly,  we define the \textbf{$p$-adic Bowen--Franks function} as $$\mathcal{L}_p(X,\alpha):=\det(\textup{Id}-A_\alpha)\in\Phi_{p,d}.$$
\end{defn}
By Lemma~\ref{lem:injectiverings}, we can regard $L_p(X,\alpha)$ and $\mathcal{L}_p(X,\alpha)$ as elements of $\Lambda_l(\Gamma)$ for all primes $l$.

Let $q\ge1$ be an integer. We say that a digraph $X$ is \textit{out-regular of degree} $q$ if $|\Out(v)|=q$ for all $v\in V(X)$. We have the following interpolation formuale.
\begin{proposition}
\label{interpolation}
    Let $\omega$ be a finite order character of $\Gamma$. Assume that $\omega$ is a character of $\Gamma_n$. Then 
    \[\omega(\mathcal{L}_p(X,\alpha))=L(X_n/X,\omega,1)^{-1}.\]
    If $X$ is in addition out-regular of degree $q$, we have
    \[\omega(L_p(X,\alpha))=q^{r}L(X_n/X,\omega,q^{-1})^{-1},\]
    where $r$ is the order $X$
\end{proposition}
\begin{proof}
The first assertion follows from Lemma~\ref{description-asmatrix} after setting $u=1$.
    
    For the second assertion, we have $D_X=q\Id$ since $X$ is assumed to be out-regular of degree $q$. Therefore
    \begin{align*}\omega(L_p(X,\alpha))&=\det(\omega(D_X-A_\alpha))=q^{r}\det(\omega(\Id-q^{-1}A_\alpha))=q^r\det(\Id-q^{-1}A_\omega)\\&=q^rL(X_n/X,\omega,q^{-1})^{-1},\end{align*}
    where the last equality follows from Lemma \ref{description-asmatrix}.
\end{proof}
\begin{remark}
In the case of undirected graphs, we have the following alternative definition of $L$-functions, known as the Artin--Ihara $L$-function. Let $Y/X$ be a Galois covering of connected undirected graphs, with Galois group $G$ that is abelian. Let $\omega$ be a character of $G$. We define
\[
L_X(u,\omega)=\frac{(1-u^2)^{\chi(X)}}{h_X(u,\omega)}\in\ZZ[\omega(G)][[u]],
\]
    where $\chi(X)$ is the Euler characteristic of $X$ and
    \[
    h_X(u,\omega)=\det(\mathrm{Id}-A_\omega u+(D_X-\mathrm{Id})u^2)\in\ZZ[\psi(G)][u].
    \]
    Here, $A_\omega$ is the adjacency matrix of $X$ twisted by $\omega$ (see \cite[\S2.4]{DV} for further details). It then follows that
    \[
    h_X(1,\omega)=\det(D_X-A_\omega).
    \]
   In other words, the $p$-adic zeta function attached to a $\Zp^d$-tower of undirected graphs interpolates the values $h_X(1,\omega)$ as $\omega$ varies over $\Gamma^\vee$.
   
It is not clear to the authors whether $L_p(X,\alpha)$ satisfies similar properties in the directed case if the digraphs are not out-regular. 
\end{remark}

\section{Iwasawa main conjectures for $\Z_p^d$-towers}
Throughout this section, we fix a $\Z^d_p$-tower of strongly connected digraphs
\[X\leftarrow X_1\leftarrow X_2\leftarrow \dots \leftarrow X_n\leftarrow \dots\]
 We also fix a prime number $l$ that may or may not be equal to $p$.

\begin{defn}
We denote by $\Div_l(X_n)$ the $\Z_l$-module $\textup{Div}(X_n)\otimes_\Z \Z_l$. Similarly, we define $\Pr_l(X_n)$, $\pic_l(X_n)$, and $\BF_l(X_n)$.  Define $\pic_l(X_\infty):=\displaystyle\varprojlim_n \pic_l(X_n)$ and $\BF_l(X_\infty):=\displaystyle\varprojlim_n\BF_l(X_n)$.
\end{defn}

\begin{defn}
    Let $m\ge n$ be non-negative integers. We define $I^m_n\subset \Z_l[\Gamma_m]$ as the kernel of the natural projection $\Z_l[\Gamma_m]\to \Z_l[\Gamma_n]$. Similarly, we define $I^\infty_n$ as the kernel of $\Z_l[[\Gamma]]\to \Z_l[\Gamma_n]$. 
\end{defn}
\begin{lemma}\label{lem:control}
For all non-negative integers $m\ge n$, there are isomorphisms
    \begin{align*}
        \pic_l(X_m)/I_n^m\pic_l(X_m)&\cong \pic_l(X_n),\\
\BF_l(X_m)/I_n^m\BF_l(X_m)&\cong \BF_l(X_n).
    \end{align*}
\end{lemma}
\begin{proof}
    Consider the following commutative diagram with exact rows
    \[\begin{tikzcd}
        0\arrow[r] &\Pr_l(X_m)\arrow[r]\arrow[d]&\Div_l(X_m)\arrow[d]\arrow[r]&\textup{Pic}_l(X_m)\arrow[d]\arrow[r]&0\\
         0\arrow[r] &\Pr_l(X_n)\arrow[r]&\Div_l(X_n)\arrow[r]&\textup{Pic}_l(X_n)\arrow[r]&0.
    \end{tikzcd}\]
    By definitions, all three vertical maps are surjective. The kernel of the middle vertical map is given by $I_n^m\Div_l(X_m)$. Hence, the first asserted isomorphism follows from the snake lemma. The second follows from the same argument by considering the following commutative diagram:
    \[\begin{tikzcd}
        0\arrow[r] &\mathrm{BF}_{X_m}\cdot\Div_l(X_m)\arrow[r]\arrow[d]&\Div_l(X_m)\arrow[d]\arrow[r]&\BF_l(X_m)\arrow[d]\arrow[r]&0\\
         0\arrow[r] &\mathrm{BF}_{X_n}\cdot\Div_l(X_n)\arrow[r]&\Div_l(X_n)\arrow[r]&\BF_l(X_n)\arrow[r]&0.
    \end{tikzcd}\]    
\end{proof}

\begin{theorem}
\label{char-udir}
Let $r$ be the order of $X_0$. The $\Lambda_l(\Gamma)$-modules ${\rm Pic}_{l}(X_{\infty})$ and $\BF_l(X_\infty)$ are isomorphic to the quotients $\Lambda^r/\left(D_X-A_\alpha\right)$ and $\Lambda^r/\left(\Id-A_\alpha\right)$, respectively.
\end{theorem}
\begin{proof}
 Recall that the vertices of $X_n$ are of the form $(v,g)$, where $v\in V(X)$, $g\in \Gamma_n$. Thus, there is a $\Zl[\Gamma_n]$-isomorphism $\Div_l(X_n)\cong \displaystyle\bigoplus_{v\in V(X)}\Z_l[\Gamma_n](v,1)$, where the action of $g\in\Gamma_n$ is given by $g\cdot(v,1)=(v,g)$. Similarly, $\Pr_l(X_n)$ is isomorphic to the $\Z_l[\Gamma_n]$-module generated by the set of divisors $\{(D_X-A_\alpha)(v,1)\mid v\in V(X)\}$. It follows that $\pic_l(X_n)$ is isomorphic to the quotient $\Zl[\Gamma_n]^r/\left(D_X-A_\alpha\right)$. Taking projective limits gives the $\Lambda$-isomorphism
    \[
    \Pic_l(X_\infty)\cong \Lambda^r/\left(D_X-A_\alpha\right),
    \]
    as desired. The proof for Bowen--Franks groups can be carried out analogously.
\end{proof}

\begin{corollary}
\label{projection-fin-level} Assume that $l\neq p$.
   Let $[\omega]\in \sR^{(l)}$. If $n$ is an integer such that $\omega\in \Gamma_n^\vee$, then  
    \[e_{[\omega]} \pic_l(X_\infty)=e_{[\omega],n}\pic_l(X_n)\]
    and
    \[e_{[\omega]} \BF_l(X_\infty)=e_{[\omega],n}\BF_l(X_n).\] 
\end{corollary}
\begin{proof}
As before, we only provide the proof for Picard groups, since the proof for Bowen--Franks groups can be carried out analogously.    By the choice of $n$, we have $e_{[\omega]}I_n^\infty\pic_l(X_\infty)=0$. Thus, we deduce
    \[e_{[\omega]}\pic_l(X_\infty)\cong e_{[\omega]}\left(\pic_l(X_\infty)/I^\infty_n\pic_l(X_\infty)\right)\cong e_{[\omega]}\pic_l(X_n)\cong e_{[\omega],n}\pic_l(X_n),\]
    where the last identity follows from the fact that $\pic_l(X_n)$ is a $\Z_l[\Gamma_n]$-module. 
\end{proof}

Through Corollary~\ref{projection-fin-level}, we can compute $\Char_{\Lambda_l(\Gamma)_{[\omega]}}\left(\pic_l(X_\infty)_{[\omega]}\right)$.

\begin{lemma}\label{lem:finite-char-ideal-non-equal-char}
Let $l$ and $p$ be distinct primes. If $\omega$ is a finite-order of $\Gamma$, then
    \[\Char_{\Lambda_l(\Gamma)_{[\omega]}}\left(\pic_l(X_\infty)_{[\omega]}\right)=(\omega(\cL_p(X,\alpha))\]
    and
     \[\Char_{\Lambda_l(\Gamma)_{[\omega]}}\left(\BF_l(X_\infty)_{[\omega]}\right)=(\omega(L_p(X,\alpha)).\]
\end{lemma}
\begin{proof}
Suppose that $\omega$ is of order $p^n$.    By Corollary \ref{projection-fin-level}, it suffices to compute $\Char_{\Lambda_l(\Gamma)_{[\omega]}}\left(e_{[\omega],n}\pic_l(X_n)\right)$. Let $r$ be the order of $X$. Let $\Lambda_n:=\Z_l[\Gamma_n]$. Consider the short exact sequence of $\Lambda_n$-modules:
    \[0\to \Lambda_n^r\to \Lambda_n^r\to \pic_l(X_n)\to 0,\]
    where the map $\Lambda_n^r\to \Lambda_n^r$ is given by the matrix $(D_X-A_\alpha)\pmod {I_n}$. As $e_{[\omega],n}$ is an idempotent, the multiplication with $e_{[\omega],n}$ is an exact functor.  Applying $e_{[\omega],_n}$ to the sequence above, we obtain the following exact sequence 
    \[0\to (e_{[\omega],n}\Lambda_n)^r\to (e_{[\omega],n}\Lambda_n)^r\to e_{[\omega],n}\pic_l(X_n)\to 0,\]
    where the map $(e_{[\omega_n],n}\Lambda_n)^r\to (e_{[\omega_n],n}\Lambda_n)^r$ is given by $e_{[\omega_n],n}(D-A_\alpha)$. Recall that $e_{[\omega],n}\Lambda_n=\Z_l[\omega(\Gamma_n)]$. Thus, we have a short exact sequence
    \[0\to\Z_l[\omega(\Gamma_n)]^r\to \Z_l[\omega(\Gamma_n)]^r\to e_{[\omega],n}\pic_l(X_n)\to 0\]
    In particular, $e_{[\omega],n}\pic_l(X_n)$ is a quadratically presented $\Z_l[\omega(\Gamma_n)]$-module. Its characteristic ideal is given by
    \begin{align*}\Char_{e_{[\omega],n}\Lambda_n}(e_{[\omega],n}\pic_l(X_n))&=(\det(e_{[\omega],n}(D-A_\alpha)))=(e_{[\omega],n}\det(D-A_\alpha))\\&=(\omega(\det(D-A_\alpha)))\subset \Z_l[\omega(\Gamma_n)],\end{align*}
    where the last equality follows from the fact that $l$ and $p$ are coprime. The proof for Bowen--Franks groups can be carried out analogously.
\end{proof}
The last remaining ingredient to complete the proof of Theorem \ref{thmA} is the following result on the non-vanishing of the $p$-adic zeta function. 
\begin{theorem}\label{thm:non-vanishing}
   If $\omega$ is a non-trivial character of $\Gamma_n$, then $$\omega(L_p(X,\alpha))\ne0.$$ In particular, $L_p(X,\alpha)\ne0$. If $X$ is furthermore out-regular of degree $q$ and $\omega$ is of order $p^n$, then $$L(X_n/X,\omega,q^{-1})\neq 0.$$
\end{theorem}
\begin{proof}
Since the non-vanishing of $\omega(L_p(X,\alpha))$ is independent of the embedding $\iota_l$, we may consider its image in $\overline{\Ql}$, where  $l\ne p$.    Recall from Lemma \ref{isomorphism-group rings} the isomorphism
    \[\Zl[\Gamma_n]\cong \bigoplus_{[\omega]\in \sR_n^{(l)}}\Z_l[\omega(\Gamma_n)].\]
    In addition, we have seen in the proof of Theorem~\ref{char-udir} that $$\Pic_l(X_n)\cong\Zl[\Gamma_n]^r/(D_X-A_\alpha).$$
    Hence, we deduce 
    \[
    \Pic_l(X_n)\cong\bigoplus_{[\omega]\in\sR_n^{(l)}}\Zl[\omega(\Gamma_n)]^r/(D_X-\omega(A_\alpha)).
    \]

    By Remark \ref{rank-picard}, the rank of the $\ZZ$-module $\Pic(X_n)$ is $1$. Thus, there exists exactly one $[\omega]$ such that $(\Z_l[\omega(\Gamma)])^r/(D-\omega(A_\alpha))$ has positive $\Zl$-rank. As this holds for all $n$, this character has to be the trivial character. In particular, for all $\omega$ that is non-trivial, $\det(D_X-\omega(A_\alpha))\ne0$, proving the first assertion of the theorem. The final assertion follows from Proposition~\ref {interpolation}.
   \end{proof}
We are now ready to prove Theorem~\ref{thmA}.
\begin{theorem}Let $(X_n)_{n\ge0}$ be a $\Zp^d$-tower of strongly connected digraphs. For all primes $\ell$, we have
\label{thm:IMC}
   \[\Char_{\Lambda_l(\Gamma)}(\pic_l(X_\infty))=(L_p(X,\alpha))\]
   and
   \[\Char_{\Lambda_l(\Gamma)}(\BF_l(X_\infty))=(\cL_p(X,\alpha)).\]
   Furthermore, ${\rm Pic}_{l}(X_{\infty})$ is a torsion $\Lambda_l(\Gamma)$-module. 
\end{theorem}
\begin{proof}
The equality of characteristic ideals follows from Theorem~\ref{char-udir} and Lemma~\ref{lem:finite-char-ideal-non-equal-char}; the torsionness of $\Pic_l(X_\infty)$ follows from Theorem~\ref{thm:non-vanishing}.
\end{proof}
\begin{remark}\label{rk:sinnott}
    For a $\Zp^d$-tower of strongly connected digraphs $(X_n)_{n\ge0}$, the $\Lambda_l(\Gamma)$-module $\Pic_l(X_\infty)$ is a Sinnott module in the sense of \cite[Definition 4.6]{BandiniLonghi}. If $\cL_p(X,\alpha)\neq 0$, the same is true for $\BF_l(X_\infty)$. 
\end{remark}

\begin{remark}
   In the case of undirected graphs, the statement in Theorem~\ref{thm:IMC} for Picard groups when $l=p$ has been proved in \cite[Theorem 6.1]{kataoka} and \cite[Theorem 5.2 and Remark 5.3]{KM1}.  
\end{remark}
\begin{example}\label{eg:unbounded BF}
    Set $d=1$. Let $X$ be a digraph with adjacency matrix
    \[A_X=\begin{pmatrix}
        2&2&0\\
        1&1&1\\
        1&2&1
    \end{pmatrix}.\]
    We can check that $X$ is strongly connected. Consider a voltage assignment $\alpha$ such that 
    \[A_\alpha=\begin{pmatrix}
        2&1+T&0\\
        1&1&1\\
        1&1+T&1
    \end{pmatrix}.\]
    We can verify using Lemma~\ref{lem:connected} that $X_n$ is strongly connected for all $n$. We have
    \[\mathcal{L}_p(X,\alpha)=\det(I-A_\alpha)=\det\begin{pmatrix}
        -1&-1-T&0\\
        -1&0&-1\\
        -1&-1-T&0
    \end{pmatrix}=0.\]
In particular, it follows from Theorem~\ref{char-udir} that $\BF_l(X_\infty)$ is not a torsion $\Lambda_\ell(\Gamma)$-module for all $l$. Furthermore, the algebraic Bowen--Franks rank $b(X_n)$ is unbounded as $n\to\infty$.
\end{example}

\section{A Sinnott--Washington type theorem}
The goal of this section is to prove Theorem~\ref{thmB}.

\begin{defn}
For each $\vec\zeta=(\zeta_1,\dots,\zeta_d)\in\mu_{p^\infty}^d$ and $a=(a_1,\dots,a_d)\in\Zp^d$, we write $a(\vec\zeta)=\zeta_1^{a_{1}}\cdots \zeta_d^{a_{d}}\in\mu_{p^\infty}$.
\end{defn}

\begin{lemma}\label{lem:eq-vects}
    Let $a,b\in\Zp^d$ be two distinct vectors, and $r\ge1$ an  integer. We have
    \[
    \#\{\vec\zeta\in\mu_{p^n}^d:a(\vec\zeta)^r=b(\vec\zeta)^r\}=O(p^{(d-1)n}).
    \]
\end{lemma}
\begin{proof}
Suppose $n$ is large enough so that $p^ra\not\equiv p^rb\mod p^n$.
    Without loss of generality, suppose that $p^ra_1\not\equiv p^rb_1\mod p^n$. Let $s=\ord_p(a_1-b_1)$. Then $r+s<n$. If $a(\vec\zeta)^{p^r}=b(\vec\zeta)^{p^r}$, then
    \[
    \zeta_1^{p^r(a_1-b_1)}=\prod_{i=2}^d\zeta_i^{p^r(b_i-a_i)}.
    \]
    For any given choice of $\zeta_2,\dots,\zeta_d\in\mu_{p^n}$, there are at most $p^{r+s}$ values of $\zeta_1\in\mu_{p^n}$ for which the equality above holds. Hence, there are at most $p^{(d-1)n+r+s}$ choices of $\vec\zeta\in\mu_{p^n}^d$ such that $a(\vec\zeta)^r=b(\vec\zeta)^r$.
\end{proof}

\begin{corollary}\label{cor:non-vanish}
    Let $F\in\Phi_{p,d}$ be a non-zero element. Let $l\ne p$ be a prime number such that $F\notin l\Phi_{p,d}$. Then
    \[
    \sharp \{\omega\in\Gamma_n^\vee:\ord_l(\omega(F))>0\}=O(p^{(d-1)n}).
    \]
\end{corollary}
\begin{proof}
Let $\bar F$ denote the image of $F$ in $\Phi_{p,d}/l\Phi_{p,d}$, and write 
\[
\bar F=\sum_{i=1}^kb_iT_1^{a_{i,1}}\cdots T_d^{a_{i,d}},
\]
where $b_i\in \FF_l^\times$ and $a_i=(a_{i,1},\dots,a_{i,d})\in\Zp^d$. 

 Suppose $\vec\zeta$ is chosen such that $a_i(\vec\zeta)^{-1}a_j(\vec\zeta)\notin\FF_l(\mu_p)$ whenever $i\ne j$. We follow the proof of \cite[Theorem~2.2]{sinnott} to show that $\ev_{\vec\zeta}\bar F\ne0$. Suppose the contrary. Then
 \[
 \ev_{\vec\zeta}\bar F=\sum_{i=1}^k\bar b_ia_i(\vec\zeta)=0.
 \]
 Let $k$ be the extension of $\FF_l$ that contains $\mu_p$ and $a_i(\vec\zeta)$ for all $i$. Let $\Tr$ denote the trace map from $k$ to $\FF_l(\mu_p)$. 
Then for all $i$,
\[
\Tr \left(a_i(\vec\zeta)^{-1}\ev_{\vec\zeta}\bar F\right)=[k:\FF_l(\mu_p)]\bar b_i=0.
\]
As $[k:\FF_l(\mu_p)]$ is a power of $p$, this implies that $\bar b_i=0$. This contradicts that $F\notin l\Phi_{p,d}$.

It remains to count the number of elements $\vec\zeta\in\mu_{p^n}^d$ such that $a_i(\vec\zeta)^{p^r}=a_j(\vec\zeta)^{p^r}$ for some $i\ne j$, where $r$ is the integer given by $\FF_l(\mu_p)\bigcap\mu_{p^\infty}=\mu_{p^r}$. By Lemma~\ref{lem:eq-vects}, the number of such elements is $O(p^{(d-1)n})$, hence the corollary follows.
\end{proof}

\begin{defn}
    Give a nonzero element $F\in\Phi_{p,d}$ and a prime number $l$, we define
    \begin{align*}
        \Supp_n(F)&=\{\omega\in \Gamma_n^\vee:\omega(F)\ne0\},\\
        \ord_{l,n}(F)&=\max(\ord_l(\omega(F)):\omega\in\Supp_n(F)).
    \end{align*}
\end{defn}

\begin{lemma}\label{lem:boundedvaluations}
    Given any non-zero $F\in\Phi_{p,d}$, $\ord_{l,n}(F)$ is bounded independently of $n$.
\end{lemma}

\begin{proof}
We consider $F$ as an element with coefficients in $\ZZ[\mu_{p^r}]$, where $\Fl(\mu_p)\bigcap\mu_{p^\infty}=\mu_{p^r}$ and prove the statement for all such elements.

Suppose $F=\sum_{i=1}^kb_iT_1^{a_{i,1}}\cdots T_d^{a_{i,d}}$. We proceed by induction on $k$. The statement clearly holds when $k=1$. Let $k\ge2$ and assume that the statement holds for all elements in $\Phi_{p,d}$ that are  defined as a sum of $k-1$ terms. As in the proof of Corollary~\ref{cor:non-vanish}, when $a_{i}(\vec\zeta)^{p^r}\ne a_{j}(\vec\zeta)^{p^r}$ for all $i \ne j$, we have $\ord_l(\ev_{\vec\zeta}F)=\min(\ord_l(b_i):1\le i \le k)$. When $a_i(\vec\zeta)^{p^r}=a_j(\vec\zeta)^{p^r}$ for some $i\ne j$, we can replace $b_iT_1^{a_{i,1}}\cdots T_d^{a_{i,d}}+b_jT_1^{a_{j,1}}\cdots T_d^{a_{j,d}}$ by $(b_i+b_ja_j(\vec\zeta)a_i(\vec\zeta)^{-1})T_1^{a_{i,1}}\cdots T_d^{a_{i,d}}$ in $F$, resulting in an element in $\Lambda_{p,l}$ with $k-1$ summands. The number of such elements with $k-1$ summands is bounded independently of $n$,  since there are finitely pairs of possible $(i,j)$ and $a_j(\vec\zeta)a_i(\vec\zeta)^{-1}\in\mu_{p^r}$.
Each of these elements has bounded $\ord_{l,n}$ by the inductive hypothesis. Therefore, $\ord_{l,n}(F)$ is also bounded.
\end{proof}

We deduce the following:
\begin{corollary}\label{cor:formula}
    Let $F=\sum_{i=1}^kb_iT_1^{a_{i,1}}\cdots T_d^{a_{i,d}}\in \Phi_{p,d}$ be a non-zero element. Then
    \[
    \sum_{\omega\in\Supp_n(F)}\ord_l(\omega(F))=p^{nd}\mu_l(F)+O(p^{(d-1)n}),
    \]
    where $\mu_l(F)=\min(\ord_l(b_i):1\le i\le k)$.    
\end{corollary}
\begin{proof}
Let $r_n=\#\{\omega\in\Gamma_n^\vee:\ord_l(\omega(F))=\mu_l(F)\}$ and $S_n=\{\omega\in\Supp_n(F):\ord_l(\omega(F))>\mu_l(F)\}$.   
Then,
\[\sum_{\omega\in\Supp_n(F)}\ord_l(\omega(F))=r_n\mu_l(F)+\sum_{\omega\in S_n}\ord_l(\omega(F)).
\]
Corollary~\ref{cor:non-vanish} tells us that $p^{dn}-r_n=O(p^{(d-1)n})$ and $\#S_n=O(p^{(d-1)n})$, and Lemma~\ref{lem:boundedvaluations} says that the summands in the second sum are $O(1)$. Hence, the corollary follows.
\end{proof}

We now give the proof of Theorem~\ref{thmB}.

\begin{theorem}\label{thm:generalizedSinnott}
    Let $(X_n)_{n\ge0}$ be a $\Zp^d$-tower of strongly connected digraphs.  Let $l\ne p$ be a prime number. Then 
    \[
    \ord_l(|\Pic_l(X_n)_\mathrm{tor}|)=p^{nd}\mu_l(L_p(X,\alpha))+O(p^{(d-1)n}).
    \] If $\mathcal{L}_p(X,\alpha)(\omega)\neq 0$ for all but finitely many $\omega\in\Gamma$, then
    \[\ord_l(\vert \BF_l(X_n)_{\textup{tor}}\vert)=p^{nd}\mu_l(\mathcal{L}_p(X,\alpha))+O(p^{(d-1)n}).\]
\end{theorem}
\begin{proof}
Let $F=L_p(X,\alpha)$. Recall from Theorem~\ref{thm:non-vanishing} that $L_p(X,\alpha)(\omega)\ne0$ for all non-trivial $\omega$. Combining this with \eqref{eq:decomp-mod}, Theorem~\ref{char-udir} and Corollary~\ref{projection-fin-level}, we deduce
\begin{align*}
\Pic_l(X_n)_\mathrm{tor}&\cong\Pic_l(X_0)_\mathrm{tor}\oplus\bigoplus_{m=1}^n\bigoplus_{\substack{[\omega]\in\sR\\\omega(\Gamma)=\mu_{p^m}}}e_{[\omega]}\Pic_l(X_\infty),\\
\ord_l(|\Pic_l(X_n)_\mathrm{tor}|)&= \ord_l(|\Pic_l(X_0)_\mathrm{tor}|)+\sum_{\omega\in\Supp_n(F)}\ord_l(\omega(F)).
\end{align*}
Hence, the first formula follows from Corollary~\ref{cor:formula}. The second formula can be proved similarly by considering
\[
\BF_l(X_n)_\mathrm{tor}\cong\BF_l(X_{n_0})_\mathrm{tor}\oplus\bigoplus_{m=n_0+1}^n\bigoplus_{\substack{[\omega]\in\sR\\\omega(\Gamma)=\mu_{p^m}}}e_{[\omega]}\BF_l(X_\infty),
\]
where $n_0$ is the maximum integer such that there exists a character $\omega$ of order $p^{n_0}$ such that $\cL_p(X,\alpha)(\omega)=0$.
\end{proof}

\begin{remark}\label{rk:comparesinnott}
In light of Remark~\ref{rk:sinnott}, we can use the theory of Sinnott modules, in particular \cite[Theorem~4.7]{BandiniLonghi}, to describe $\ord_l(|\Pic_l(X_n)_\mathrm{tor}|)$ (or $\ord_l(\vert \BF_l(X_n)_{\textup{tor}}\vert)$) in terms of a sequence of integers $(t_m)_{m\ge0}$. Theorem~\ref{thm:generalizedSinnott} shows that this sequence converges to a constant, i.e. $\mu_l(L_p(X,\alpha))$ (or $\mu_l(\cL_p(X,\alpha))$), as $m\to\infty$.

When $d=1$, the hypothesis $\mathcal{L}_p(X,\alpha)(\omega)\neq 0$ for all but finitely many $\omega\in\Gamma$ is equivalent to $\cL_p(X,\alpha)\ne0$. However, the former condition is in general weaker than the latter when $d>1$.

\end{remark}

\section{Defects of digraphs}
Recall from Remark \ref{rank-picard} that the rank of the Picard group of strongly connected digraphs is equal to 1. In particular, it stays bounded along the $\Z_p^d$-towers we study in this article. We have seen in Example~\ref{eg:unbounded BF} that this is not the case for the algebraic Bowen--Franks rank $b(X_n)$.  In this section, we study the growth of $b(X_n)$ in more detail. We will also define and study its analytic counterpart.

\begin{defn}
    Let $X$ be a digraph. We define the \textbf{analytic Bowen--Franks rank} of $X$ as  $a(X)=-\ord_{u=1}(Z_X(u))$. The \textbf{defect} of $X$ is defined as $\delta(X)=a(X)-b(X)$. 
\end{defn}
The fact that the algebraic multiplicity of an eigenvalue of a square matrix is greater than equal to its geometric multiplicity means that $a(X)\ge b(X)$. In other words, $\delta(X)\ge0$.
One can regard $\delta(X)$ as the failure for a "BSD-type conjecture" to hold.    

If the adjacency matrix $A_X$ of $X$ is diagonalizable (e.g., when $X$ is undirected, in which case $A_X$ is symmetric), then $a(X)=b(X)$ and $\delta(X)=0$. For a general digraph, it is possible for $\delta(X)$ to be non-zero, as illustrated by the following example.

\begin{example}
\label{example:pos-defect}
    Let $X$ be a digraph of order 3, with the adjacency matrix
    \[A_X=\begin{pmatrix}
        1&0&1\\
        1&2&1\\
        1&1&2
    \end{pmatrix}.\]
    Lemma \ref{description-asmatrix} implies
    \[Z_X(u)=\det(\mathrm{Id}-uA_X)^{-1}.\]
    Note that $1$ is an eigenvalue of $A_X$ with algebraic multiplicity $2$. Thus, $a(X)=2$. The kernel of $\mathrm{Id}-A_X$ is  one-dimensional, which implies that $b(X)=1$. Therefore, $\delta(X)=2-1=1>0$.
\end{example}

\begin{remark}
    Let $Y/X$ be a covering of digraphs. There is a natural surjection $\BF(Y)\to \BF(X)$. In particular, $b(Y)\ge b(X)$. In addition, the Artin formalism (Theorem~\ref{Artin-formalism}) implies that $a(Y)\ge a(X)$. 
\end{remark}

\begin{defn}
    Let $(X_n)_{n\ge0}$ be a $\Zp^d$-tower of strongly connected digraphs. Let $\omega \in \Gamma_n^\vee$. We define $a(X_n,\omega)=-\ord_{u=1}L(X_n/X,\omega,u)$. We define $b(X_n,\omega)$ as the $\ZZ[\omega(\Gamma_n)]$-rank of the cokernel of $\textup{Id}-\omega(A_\alpha)$ as a linear map on $\ZZ[\omega(\Gamma_n)]^r\to \ZZ[\omega(\Gamma_n)]^r$.
\end{defn}
\begin{lemma}
\label{formula-a,b}
    We have the following identities
    \begin{align*}
        a(X_n)&=\sum_{\omega\in \Gamma_n^\vee}a(X_n,\omega),\\
        b(X_n)&=\sum_{\omega\in \Gamma_n^\vee}b(X_n,\omega).
    \end{align*}
\end{lemma}
\begin{proof}
    The equation for the analytic ranks follows from the Artin formalism (Theorem \ref{Artin-formalism}). It remains to prove the equation for the algebraic ranks. 

   Let $l$ be a prime number distinct from $p$. For each $\omega\in\Gamma_n^\vee$, let $f_\omega$ denote $[\Ql(\omega(\Gamma_n)):\Ql]$. There are $f_\omega$ characters in the equivalence class $[\omega]$ and
   \[
   b(X_n,\omega)=b(X_n,\omega')\ \forall \omega'\in[\omega],
   \]
   and  
   \begin{equation}
   \label{eq:b-formula}\sum_{\omega'\in[\omega]}b(X_n,\omega')=f_\omega\cdot  b(X_n,\omega)=\rank_{\Zl}\Zl[\omega(\Gamma_n)]^r/(\mathrm{Id}-\omega(A_\alpha)).
   \end{equation}

 Recall from Lemma~\ref{isomorphism-group rings} that
    \[ \varphi_n:\Zl[\Gamma_n]\cong\bigoplus_{[\omega]\in\sR_n^{(l)}}\Zl[\omega(\Gamma_n)].\]
Further, we have seen in the proof of Theorem~\ref{char-udir} that there is an isomorphism $\Div_l(X_n)\cong\displaystyle \bigoplus_{v\in V(X)}\Zl[\Gamma_n](v,1)$, which gives
    \[\BF_l(X_n)\cong \Zl[\Gamma_n]^r/(\textup{Id}-A_\alpha).\] 
    
    For each $[\omega]\in\sR_n^{(l)}$,  define the map $\psi_{[\omega]}$ on $\Zl[\omega(\Gamma_n)]^r$ given by the matrix $\textup{Id}-\omega(A_\alpha)$. Let $\psi_n=\displaystyle\bigoplus_{[\omega]\in\sR_n^{(l)}}\psi_{[\omega]}$ be the corresponding map on $\displaystyle\bigoplus_{m\le n}\Zl[\omega_m(\Gamma_n)]^r$. We have the following commutative diagram
    \[\begin{tikzcd}
        \Zl[\Gamma_n]^r\arrow[r,"\textup{Id}-A_\alpha"]\arrow[d,"\varphi_n"]&\Zl[\Gamma_n]^r\arrow[d,"\varphi_n"]\\
        \left(\bigoplus_{m\le n}\Zl[\omega_m(\Gamma_n)]\right)^r\arrow[r, "\psi_n"]&\left(\bigoplus_{m\le n}\Zl[\omega_m(\Gamma_n)]\right)^r.
    \end{tikzcd}\]
    In particular, \[b(X_n)=\rank_{\Zl}(\coker(\psi_n))=\sum_{[\omega]\in \sR_n^{(l)}}\rank_{\Zl}(\coker(\psi_{[\omega]}))=\sum_{\omega\in \Gamma_n^\vee}b(X_n,\omega),\]
    where the last equality follows from \eqref{eq:b-formula}.
\end{proof}
 Lemma \ref{formula-a,b} can be regarded as an analogue of the Artin formalism for the Bowen--Franks ranks $a(X_n)$ and $b(X_n)$. 

We are now ready to prove Theorem~\ref{thmC}. 
\begin{theorem}
\label{thm.growth-delta}
    Let $(X_n)_{n\ge0}$ be a $\Z_p^d$-tower of strongly connected digraphs. For all $n\ge1$,
    \[\delta(X_n)\ge \delta(X_{n-1}).\]
\end{theorem}

\begin{proof}
    Let $\omega \in{\Gamma_n^\vee}$. Note that  $b(X_n,\omega)$ is the geometric multiplicity of the eigenvalue $1$ of $\omega(A_\alpha)$. As $a(X_n,\omega)$ is the algebraic multiplicity of the eigenvalue $1$ of $\omega(A_\alpha)$, we have $a(X_n,\omega)\ge b(X_n,\omega)$.  
    Lemma \ref{formula-a,b} allows us to express $\delta(X_n)$ as
    \begin{align*}&\sum_{\omega\in \Gamma_n^\vee}a(X_n,\omega)-\sum_{\omega\in \Gamma_n^\vee}b(X_n,\omega)\\
    &=\sum_{\omega \in \Gamma_{n-1}^\vee}a(X_{n},\omega)-\sum_{\omega\in \Gamma_{n-1}^\vee}b(X_n,\omega)+\sum_{\substack{\omega\in \Gamma_n^\vee\\\ord(\omega)=p^n}}a(X_{n},\omega)-\sum_{\substack{\omega\in \Gamma_n^\vee\\\ord(\omega)=p^n}}b(X_n,\omega)\\
    &=\sum_{\omega \in \Gamma_{n-1}^\vee}a(X_{n-1},\omega)-\sum_{\omega\in \Gamma_{n-1}^\vee}b(X_{n-1},\omega)+\sum_{\substack{\omega\in \Gamma_n^\vee\\\ord(\omega)=p^n}}a(X_{n},\omega)-\sum_{\substack{\omega\in \Gamma_n^\vee\\\ord(\omega)=p^n}}b(X_n,\omega)\\
    &=\delta(X_{n-1})+\sum_{\substack{\omega\in \Gamma_n^\vee\\\ord(p^n)}}\left(a(X_n,\omega)-b(X_n,\omega)\right)\\
    &\ge \delta(X_{n-1}), \end{align*}
as desired.
\end{proof}

\subsection{The special case of $d=1$}

In this subsection, we study the special case of $\Zp$-towers.
We will first prove a sufficient condition for $a(X_n)$ and $b(X_n)$ to be bounded independently of $n$. 
\begin{lemma}\label{lem:bounded-defect}
    Let $(X_n)_{n\ge0}$ be a $\Z_p$-tower of strongly connected digraphs. Assume that $\cL(X,\alpha)\neq 0$. 
    Then $a(X_n)$ and $b(X_n)$ are uniformly bounded.  In particular, $\delta(X_n)$ is bounded independently.
\end{lemma}
\begin{proof}
    As $a(X_n)\ge b(X_n)$, it suffices to prove the boundedness of $a(X_n)$. By the Artin formalism (Theorem \ref{Artin-formalism}), we have
    \[Z_{X_n}(u)=\prod_{\omega \in\Gamma_n^\vee}L(X_n/X,\omega,u).\]
    Recall that $L(X_n/X,\omega,1)^{-1}=\omega(\cL_p(X,\alpha))$ by Theorem~\ref{interpolation}. Considering  $\cL_p(X,\alpha)$ as an element of $\Zp[[T-1]]$, we can write $\cL_p(X,\alpha)=g(T-1)p^\mu u(T-1)$ for some integer $\mu\ge0$, a distinguished polynomial $g(T-1)$ and a unit $U(T-1)$. Therefore, $L(X_n/X,\omega,1)^{-1}=\omega(g(T-1))p^\mu u_\omega$ for some $p$-adic unit $u_\omega$. Thus,  $L(X_n/X,\omega,u)$ has a pole at $u=1$ if and only if $\omega(g(T-1))=0$. As $g(T-1)$ is a polynomial, this happens only for finitely many characters $\omega$, which implies that $a(X_n)$ is uniformly bounded, as desired. 
\end{proof}

The following example illustrates that it is necessary to assume the non-vanishing of $\cL_p(X,\alpha)$.

\begin{example}
Suppose $p\ne2$ and consider the digraph $X$ with the following adjacency matrix
    \[A_X=\begin{pmatrix}
     1&0&2\\
     2&3&0\\
     2&2&3
    \end{pmatrix}.\]
 Define a voltage assignment $\alpha$ on $X$ such that 
    \[A_\alpha=\begin{pmatrix}
        1&0&1+T\\
        1+T&2+T&0\\
        1+T&1+T&2+T
    \end{pmatrix}.\]
We have 
\[\cL_p(X,\alpha)=\det\begin{pmatrix}0&0&-(1+T)\\
-(1+T)&-(1+T)&0\\
-(1+T)&-(1+T)&-(1+T)
\end{pmatrix}=0\]
The
matrix $\mathrm{Id}-A_\alpha$ has rank $2$ and its characteristic polynomial is given by $\lambda^2(\lambda+2+2T)$. For all $\omega\in\Gamma^\vee$, we have $\omega(1+T)\ne0$. Thus, $a(X_n,\omega)=2$ and $b(X_n,\omega)=1$. It follows from Lemma~\ref{formula-a,b} that
\[a(X_n)=2p^n,\quad b(X_n)=p^n.\]
It follows that $\delta(X_n)=p^n$ is an increasing and unbounded sequence. 
\end{example}

In the remainder of this section we consider the special case of constant voltage assignments originally studied in \cite{constant-voltage}.
\begin{lemma}
\label{defec-constant-voltage}
Let $X$ be a strongly connected digraph of order $r$. Let $\alpha$ be the $\Zp$-valued voltage assignment that sends all edges to an element $\tau\in\Zp^\times$. Let $(X_n)_{n\ge0}$ be the $\Zp$-tower that arises from $\alpha$. If $n$ is an integer such that $\varphi(p^n)>r$ where $\varphi$ denotes Euler's totient function, then $\delta(X_n)=\delta(X_{n-1})$. In particular, $\delta(X_n)$ is uniformly bounded. Furthermore, It is constant for all $n$ if $p-1>r$. 
\end{lemma}
\begin{proof}
    Let $\omega$ be a character of order $p^n$. Then $a(X_n,\omega)>0$ if and only if $1$ is an eigenvalue of $\omega(\tau)A_X=\omega(\tau A_X)=\omega(A_\alpha)$. This is only the case if $\omega(\tau)^{-1}$ is an eigenvalue of $A_X$. As the degree of the characteristic polynomial of $A_X$ is $r$, we must have $r\ge[\Q(\zeta_{p^{n}}):\Q]=\varphi(p^n) $. Thus, whenever $\varphi(p^n)>r$, we have $a(X_n,\omega)=b(X_n,\omega)=0$. Hence, it follows from Lemma~\ref{formula-a,b} that  $\delta(X_n)=\delta(X_{n-1})$.
\end{proof}

\begin{example}
    Let $X$ be a directed cycle with $r$ vertices. Let $X_n$ be the $n$-th layer of the constant $\Z_p$-tower of $X$. Then $X_n$ is a directed cycle with $p^nr$-vertices. The characteristic polynomial of $A_{X_n}$ is $X^{p^nr}-1$. Thus, $a(X_n)=b(X_n)=1$ and $\delta(X_n)=0$.
\end{example}
Our calculations lead us to formulate the following conjecture: 
\begin{conjecture}\label{conj:constant}
    Let $(X_n)_{n\ge0}$ be a constant $\Z_p$-tower of strongly connected digraphs. Then $\delta(X_n)=\delta(X)$ for all $n$.
\end{conjecture}
The following example shows that for non-constant voltage assignments, the defect can grow along $\Z_p$-towers even if it is uniformly bounded.
\begin{example} Consider a digraph $X$ with the following adjacency matrix
\[A_X=\begin{pmatrix}
    1&1&0&5\\
    1&6&1&0\\
    0&6&6&1\\
    1&0&1&1
\end{pmatrix}.\] It can be verified that $X$ is  strongly connected. Let $\alpha$ be a voltage assignment on $X$ such that 
    \[
A_\alpha=
\begin{pmatrix}
1  & 1 & 0&T^4+T^3+T^2+T+1\\
1 &  T^4+T^3+T^2+T+2 & 1&0\\
0 &T^4+T^3+T^2+T+2 &  T^4+T^3+T^2+T+2&1\\
1&0&1&1
\end{pmatrix}
\]

This produces a strongly connected $\Z_5$-tower. We can check that $1$ is not an eigenvalue of $A$. In particular, $\delta(X)=0$. 

Let $\zeta_0$ be a primitive fifth root of unity. For $1\le k\le 4$. the matrix $A_\alpha(\zeta^k_0)$ is equal to
\[\begin{pmatrix}
    1&1&0&0\\
    1&1&1&0\\
    0&1&1&1\\
    1&0&1&1
    
\end{pmatrix}.\]
Note that $\lambda=1$ is an eigenvalue of algebraic multiplicity $2$ and geometric multiplicity $1$.  Thus, for every character $\omega$ of order 5, we have $a(X_1,\omega)=2$ and $b(X_1,\omega)=1$. Thus, $\delta(X_1)=4$. In fact, given a complex number $\zeta$, the matrix $A_\alpha|_{T=\zeta}$ admits 1 as an eigenvalue if and only if  $\zeta$ is a primitive $5$-th root of unity. Thus, $a(X_n)=a(X_1)$ for all $n\ge1$. In particular, $a(X_n)$ is bounded as $n\to\infty$. 
\end{example}

\subsection{Application to isogeny graphs}
In this subsection, let $p,q$ and $l$ be distinct prime numbers. 

\begin{defn}\label{def:isogeny}
     We denote by $\mathcal{X}(q)$ the digraph whose vertices are isomorphism classes of supersingular elliptic curve $E/\mathbb{F}_{l^2}$ and whose edges are given by equivalence classes of $q$-isogenies, where two isogenies are deemed equivalent if they differ by an automorphism. Let $\mathcal{X}_0$ be a fixed connected component of $\mathcal{X}(q)$.  Let $\mathcal{X}_n$ be the $n$-th layer of the constant $\Z_p$-tower of $\mathcal{X}_0$. 
\end{defn}

The purpose of this subsection is to prove that Conjecture~\ref{conj:constant} holds for the tower $(\cX_n)_{n\ge0}$. 
We will make use of an auxiliary $\Zp$-tower, which we define below.
\begin{defn}
     For an integer $n\ge1$, let $\cX(q,n)$ be the digraph where the vertices are $(E,\zeta)$, where $E\in V(\cX(q))$ and $\zeta$ is a primitive $p^n$-th root of unity and the set of edges from $(E,\zeta)$ to $(E',\zeta')$ corresponds to the $q$-isogenies from $E$ to $E'$ when $\zeta'=\zeta^q$, and there are no edges when $\zeta'\ne \zeta^q$.
     \end{defn}

     As discussed in \cite[Example 5.7]{LM2}, $\cX(q,n)$ is the derived digraph of $\cX(q)$ for the voltage assignment that sends all edges to the image of $q$ in $(\ZZ/p^n\ZZ)^\times$. 
     In particular, $\cX(q,n)$ is a $\phi(p^n)$-sheeted cover of $\cX(q)$. Given an elliptic curve $E$ and a subgroup $H$ of $\GL_2(\ZZ/p^n\ZZ)$, a level $H$ structure on $E$ is an isomorphism of $\phi:(\ZZ/p^n\ZZ)^2\to E[p^n]$ considered up to composition with elements of $H$, as given in \cite[Definition 1.1]{codogni-lido}. Considering the Weil pairing of $\phi(1,0)$ with $\phi(0,1)$, we deduce that the choice of a primitive $p^n$-th root of unity $\zeta$ in $(E,\zeta)$ is  a  $\SL_2(\ZZ/p^n\ZZ)$ level structure on $E$.
     
     There is a natural covering $\cX(q,1)\to\cX(q)$.
     In what follows, we fix a strongly connected component $\cY_0$ of $\cX(q,1)$ that lies inside the pre-image of $\cX_0$. For each $n\ge0$, let $\cY_n$ denote a strongly connected component of $\cX(q,n+1)$ that lies inside the pre-image of $\cY_0$ under the projection map $\cX(q,n+1)\to\cX(q,1)$. There exists an index $n_0$ such that $(\cY_n)_{n\ge n_0}$ is a $\Zp$-tower of strongly connected digraph. To simplify notation, we will assume that $n_0=0$. We remark that the proofs we present below will work in the same way for $n_0>0$. Note that $(\cY_n)_{n\ge 0}$ arises from the constant voltage assignment that sends all edges of $\cY_0$ to $q\in\Zp$, and there is a natural covering $\cY_n\to \cX_n$.

\begin{proposition}\label{prop:Yn}
For all $n\ge0$ and all non-trivial characters $\omega$ in $\Gamma_n^\vee$, we have    $a(\mathcal{Y}_n,\omega)=0$. In particular, $\delta(\mathcal{Y}_n)=\delta(\mathcal{Y}_0)$ for all $n$.
\end{proposition}
\begin{proof}
As in  the proof of Lemma~\ref{defec-constant-voltage}, in order to show that $\delta(\mathcal{Y}_n)=\delta(\mathcal{Y}_0)$ for all $n\ge0$, it suffices to show that the eigenvalues of the adjacency matrix $A_{\cY_0}$ are not primitive $p$-power roots of unity.

Let $k$ be the order of $q$ in $(\ZZ/p\ZZ)^\times/\det(\SL_2(\ZZ/p\ZZ))=(\ZZ/p\ZZ)^\times$, and $k'$ the smallest positive integer such that $q^{k'}\mathrm{Id}\in\SL_2(\ZZ/p\ZZ)$, i.e., the order of $q^2$ in $(\ZZ/p\ZZ)^\times$. In particular, both $k$ and $k'$ divide $p-1$.
We recall from \cite[Theorem 1.6]{codogni-lido} that if $\lambda$ is an eigenvalue of $A_{\cY_0}$, then $\lambda$ is of the form $(q+1)\zeta$, where $\zeta\in\mu_{p-1}$, or $\lambda$ is a complex number with angle in $\ZZ\pi/k'$. Therefore, none of the eigenvalues of $A_{\cY_0}$ are primitive $p$-power roots of unity, as desired.
\end{proof}

We conclude with the following theorem.

\begin{theorem}
\label{cor.delta-isogeny}
Let $p,q$ and $l$ be distinct prime numbers. 
    Let $\cX_n$ be defined as in Definition~\ref{def:isogeny}. Then \[\delta(\mathcal{X}_n)=\delta(\mathcal{X}_0)\]
    for all $n\ge 0$. 
\end{theorem}
\begin{proof}Let $\omega\in\Gamma_n^\vee$ be a non-trivial character.
Since there is a covering $\mathcal{Y}_n\to\mathcal{X}_n$, we have. $a(\mathcal{Y}_n,\omega)\ge a(\mathcal{X}_n,\omega)$. Thus, it follows from $a(\mathcal{Y}_n,\omega)= a(\mathcal{X}_n,\omega)=0$, which implies that $\delta(\cX_n)=\delta(\cX_0)$. 
    \end{proof}
\bibliographystyle{alpha} 
\bibliography{references}

\begin{thebibliography}{Gon21}

\bibitem[BF77]{BF}
Rufus Bowen and John Franks.
\newblock Homology for zero-dimensional nonwandering sets.
\newblock {\em Ann. of Math. (2)}, 106(1):73--92, 1977.

\bibitem[BL23]{BandiniLonghi}
Andrea Bandini and Ignazio Longhi.
\newblock The algebra {$\mathbb{Z}_\ell[[\mathbb{Z}_p^d]]$} and applications to {I}wasawa theory.
\newblock 2023.
\newblock preprint, arXiv:2312.04666.

\bibitem[Bou98]{Bourbaki_commutative_algebra}
Nicolas Bourbaki.
\newblock {\em Commutative algebra. {C}hapters 1--7}.
\newblock Elements of Mathematics (Berlin). Springer-Verlag, Berlin, 1998.
\newblock Translated from the French, Reprint of the 1989 English translation.

\bibitem[CL23]{codogni-lido}
Giulio Codogni and Guido Lido.
\newblock Spectral theory of isogeny graphs, 2023.
\newblock preprint, arXiv:2308.13913.

\bibitem[CM81]{cuoco-monsky}
Albert~A. Cuoco and Paul Monsky.
\newblock Class numbers in {${\bf Z}^{d}\sb{p}$}-extensions.
\newblock {\em Math. Ann.}, 255(2):235--258, 1981.

\bibitem[DV23]{DV}
Sage DuBose and Daniel Valli\`eres.
\newblock On {$\Bbb Z^d_\ell$}-towers of graphs.
\newblock {\em Algebr. Comb.}, 6(5):1331--1346, 2023.

\bibitem[Fra84]{franks}
John Franks.
\newblock Flow equivalence of subshifts of finite type.
\newblock {\em Ergodic Theory Dynam. Systems}, 4(1):53--66, 1984.

\bibitem[Gon21]{gonet-thesis}
Sophia~R. Gonet.
\newblock {\em Jacobians of {F}inite and {I}nfinite {V}oltage {C}overs of {G}raphs}.
\newblock ProQuest LLC, Ann Arbor, MI, 2021.
\newblock Thesis (Ph.D.)--The University of Vermont and State Agricultural College.

\bibitem[Kat24]{kataoka}
Takenori Kataoka.
\newblock Fitting ideals of {J}acobian groups of graphs.
\newblock {\em Algebr. Comb.}, 7(3):597--625, 2024.

\bibitem[KM24]{KM1}
Sören Kleine and Katharina Müller.
\newblock On the growth of $\mathbf{Z}_p^l$-voltage covers of graphs.
\newblock {\em Algebraic Combinatoric}, 7(4):1011--1038, 2024.

\bibitem[KS00]{ks}
Motoko Kotani and Toshikazu Sunada.
\newblock Zeta functions of finite graphs.
\newblock {\em J. Math. Sci. Univ. Tokyo}, 7(1):7--25, 2000.

\bibitem[LM21]{lindmarcus}
Douglas Lind and Brian Marcus.
\newblock {\em An introduction to symbolic dynamics and coding}.
\newblock Cambridge Mathematical Library. Cambridge University Press, Cambridge, second edition, 2021.

\bibitem[LM24]{constant-voltage}
Antonio Lei and Katharina M\"uller.
\newblock On $\mathbb{Z}_p$-towers of graph coverings arising from a constant voltage assignment, 2024.
\newblock to appear in Glasgow Math. J.

\bibitem[LM25]{LM2}
Antonio Lei and Katharina M\"uller.
\newblock On towers of isogeny graphs with full level structures.
\newblock {\em Res. Math. Sci.}, 12(1):Paper No. 4, 2025.

\bibitem[LV23]{leivallieres}
Antonio Lei and Daniel Valli\`eres.
\newblock The non-{$\ell$}-part of the number of spanning trees in abelian {$\ell$}-towers of multigraphs.
\newblock {\em Res. Number Theory}, 9(1):Paper No. 18, 16, 2023.

\bibitem[MV23]{vallieres2}
Kevin McGown and Daniel Valli\`eres.
\newblock On abelian {$\ell$}-towers of multigraphs {II}.
\newblock {\em Ann. Math. Qu\'{e}.}, 47(2):461--473, 2023.

\bibitem[MV24]{vallieres3}
Kevin McGown and Daniel Valli\`eres.
\newblock On abelian {$\ell$}-towers of multigraphs {III}.
\newblock {\em Ann. Math. Qu\'{e}.}, 48(1):1--19, 2024.

\bibitem[Sin87]{sinnott}
W.~Sinnott.
\newblock On a theorem of {L}. {W}ashington.
\newblock Number 147-148, pages 209--224, 344. 1987.
\newblock Journ\'ees arithm\'etiques de Besan\c con (Besan\c con, 1985).

\bibitem[Ter11]{terras}
Audrey Terras.
\newblock {\em Zeta functions of graphs}, volume 128 of {\em Cambridge Studies in Advanced Mathematics}.
\newblock Cambridge University Press, Cambridge, 2011.
\newblock A stroll through the garden.

\bibitem[Val21]{vallieres}
Daniel Valli\`eres.
\newblock On abelian {$\ell$}-towers of multigraphs.
\newblock {\em Ann. Math. Qu\'{e}.}, 45(2):433--452, 2021.

\bibitem[VL20]{VL}
J.~J.~P. Veerman and R.~Lyons Lyons.
\newblock {A Primer on Laplacian Dynamics in Directed Graphs}.
\newblock {\em Nonlinear Phenomena in Complex Systems}, 23(2):196--206, 2020.

\end{thebibliography}
\end{document}